\documentclass[a4paper,12pt]{article}

\usepackage[margin=3cm,footskip=1cm]{geometry}

\usepackage[T1]{fontenc}
\usepackage{amsmath,amssymb,amsthm,bm,mathtools,xspace,booktabs,url}
\usepackage{graphicx,etoolbox}
\usepackage{color}

\usepackage{multirow,subfigure}
\newcommand{\e}[1]{\frac{#1}3\eta_{\max}}

\usepackage{soul}
\usepackage[shortlabels]{enumitem}
\setlist{
  listparindent=\parindent,
  parsep=0pt,
}

\usepackage[round]{natbib}

\usepackage[raggedright]{titlesec}

\numberwithin{equation}{section}

\theoremstyle{plain} 
\newtheorem{theorem}{Theorem}[section]
\newtheorem{lemma}[theorem]{Lemma}
\newtheorem{proposition}[theorem]{Proposition}

\theoremstyle{definition} 
\newtheorem{example}[theorem]{Example}
\newtheorem{remark}[theorem]{Remark}

\usepackage{authblk}

\makeatletter
\newcommand\CorrespondingAuthor[1]{%
  \begingroup%
  \def\@makefnmark{}%
  \footnotetext{Corresponding author: #1}%
  \endgroup%
}
\makeatother

\makeatletter
\renewenvironment{abstract}{%
  \small%
  \begin{center}%
    \bfseries \abstractname\vspace{-.5em}\vspace{\z@}%
  \end{center}%
  \quote%
}{\endquote}
\makeatother

\usepackage{siunitx}

\usepackage[all]{onlyamsmath}

\usepackage{lipsum}

\usepackage[utf8]{inputenc}
\usepackage{hyperref}
\usepackage{fixme}

\usepackage{subfigure}
\usepackage{float}
\makeatletter
\DeclareRobustCommand*\subref{\@ifstar\sf@@subref\sf@subref}
\makeatother

\DeclareMathOperator\Var{Var}
\DeclareMathOperator\Cov{Cov}
\DeclareMathOperator\E{E}
\DeclareMathOperator\PP{P}

\newcommand{\Z}{\mathbb{Z}}
\newcommand{\R}{\mathbb{R}}
\newcommand{\N}{\mathbb{N}}

\renewcommand{\i}{\mathbf{i}}

\newcommand{\bY}{\mathbf{Y}}
\newcommand{\bZ}{\mathbf{Z}}
\newcommand{\bC}{\mathbf{C}}
\newcommand{\bL}{\mathbf{\Lambda}}
\newcommand{\bQ}{\mathbf{Q}}
\newcommand{\bN}{\mathbf{N}}

\newcommand{\mR}{{\mathcal{R}}}
\newcommand{\mI}{{\mathcal{I}}}
\newcommand{\mt}{{\tilde m}}
\newcommand{\app}{\mathrm{app}}

\newcommand{\sign}{\mathrm{sign}}

\begin{document}
\title{Fast and exact simulation of complex-valued stationary Gaussian processes through embedding circulant matrix}

\author[1]{Jean-Francois Coeurjolly}
\affil[1]{Laboratory Jean Kuntzmann, Grenoble Alpes University,
  France, \texttt{Jean-Francois.Coeurjolly@upmf-grenoble.fr}}

\author[2]{Emilio Porcu}
\affil[2]{Department of Mathematics, Technical University Federico Santa Maria, Chile,  \texttt{emilio.porcu@usm.cl}}


\date{\today}

\maketitle

\begin{abstract}

This paper is concerned with the study of the embedding circulant matrix method  to simulate stationary complex-valued Gaussian sequences. The method is, in particular, shown to be well-suited to generate circularly-symmetric stationary Gaussian processes. We provide simple conditions on the complex covariance function ensuring the theoretical validity of the minimal embedding circulant matrix method.
We show that these conditions are satisfied by many examples and illustrate the algorithm. In particular, we present a simulation study involving the circularly-symmetric fractional Brownian motion, a model introduced in this paper.\\
  

\end{abstract}

\section{Introduction} 
\label{sec:introduction}

Complex-valued Gaussian processes have emerged in a wide variety of domains and applications, such as physics, engineering sciences, signal processing (see e.g. \citet{curtis:85,dunmire:00,amblard:gaeta:lacoume:96b}), digital communication \citep{lee:messerschmitt:94}, climate modelling \citep{tobar:turner:15}. The present paper focusses on fast and exact simulation of a discretized sample path from a stationary complex-valued Gaussian sequence. By fast, we mean that the method can be applied for very large sample sizes, and by exact, we mean that the output vector has the expected covariance matrix.

The simulation of stationary Gaussian sequences is an important problem which has generated an important literature. Amongst available methods, the embedding circulant matrix method is probably the most popular as a very efficient alternative to methods based on the Cholesky decomposition. Introduced by~\citet{davies:harte:87}, the method has been popularized by \citet{wood:chan:94}. The main idea is to embed the covariance matrix, say $\mathbf \Gamma$, of the stationary sequence to be simulated, into a circulant matrix $\bC$. Unlike the diagonalization of $\mathbf \Gamma$, which can be computationally intensive for large sample sizes, the diagonalization of $\bC$ can be efficiently performed using the Fast Fourier Transform since, as a circulant matrix, $\bC$ is diagonalizable in the Fourier basis. 
For $n$ being the sample size, the computational cost of the embedding circulant matrix method is $\mathcal O(n\log n)$, which considerably mitigates the computational burden of Cholesky decomposition methods, being of the order $\mathcal O(n^2)$ for Teoplitz matrices.

A non trivial requirement of circulant embedding method is that the matrix $\bC$ must be non-negative. This problem has also been the focus of several papers, and we especially refer to~\citet{dietrich:newsam:97} and~\citet{craigmile:03} for simple and verifiable conditions on the covariance function, ensuring the non-negativeness of $\bC$. It is noticeable that the combination of these two works covers elaborate models, such as the fractional Brownian motion (see e.g. \citet{coeurjolly:00} and the FARIMA model (see e.g. \citet{brockwell:davis:87}).

Since the 90's, the embedding circulant matrix method has been extended in many directions. \citet{chan:wood:99} extended their algorithm to generate stationary univariate or multivariate random fields, as well as multivariate time series. This technical paper has recently been revisited by~\citet{helgason:pipiras:abry:11} for multivariate time series. In particular, the authors provided conditions ensuring the validity of the embedding circulant matrix method. In the context of random fields, the method has also known many developments by e.g.~\citet{stein:02}, \citet{gneiting:12}, \citet{ davies:bryant:13}, \citet{helgason:pipiras:abry:14} among others.

To generate a complex-valued stationary Gaussian sequence with given complex-valued covariance function, one can obviously simulate the corresponding real-valued bivariate  stationary Gaussian process, and take the first (resp. second) component to define the real (resp. imaginary) part of the complex-valued stationary sequence  to be simulated. This strategy, however, does not exploit the fact that any circulant Hermitian matrix can still be diagonalized using the Fourier basis. \citet{percival:06} indeed noticed this, and proposed an algorithm to generate a stationary complex-valued sequence with given complex-valued covariance matrix $\mathbf \Gamma$.

In order to characterize a complex-valued Gaussian process, covariance and pseudo covariance are both needed (see Section~\ref{sec:background} for more details). This paper digs into the algorithms proposed by \citet{wood:chan:94} and~\citet{percival:06}. Special emphasis is put on understanding the consequences on the control of the pseudo-covariance matrix. In addition, following the works by~\citet{dietrich:newsam:97} and~\citet{craigmile:03}, we provide  conditions which ensure the validity of the mimnimal embedding circulant matrix method in the complex case. 

The rest of the paper is organized as follows.
Section~\ref{sec:background} presents our main notation, provides a short background and details several examples. The simulation algorithms as well as an approximation, in the case where $\bC$ is negative, are presented in Section~\ref{sec:simulation}. Section~\ref{sec:positivedefiniteness} is focused on the theoretical validation of the embedding circulant matrix method for complex processes. We return to the examples in Section~\ref{sec:examples}. We apply our theoretical conditions and illustrate the algorithms. In this section, we also use the simulation algorithm to compare several confidence intervals for the Hurst parameter of the circularly complex fractional Brownian motion, a model introduced in Section~\ref{sec:background}. Finally, proofs of our results are postponed to Appendix.

\section{Background and notation} 
\label{sec:background}

We denote $Z=\{ Z(t)\}_{t\in S}$ a strictly stationary and complex-valued Gaussian process with index set $S$ being either the real line, $\R$ or the set of integers, $\Z$, or subsets of them. Since $Z$ is complex-valued, it can be uniquely written as $Z(t)=Z_\mR(t)+ \i Z_\mI(t)$, for $t\in S$, with $\i$ being the complex number verifying $\i^2=-1$. In particular, $Z_\mR$ and $Z_\mI$ are called real and imaginary parts, respectively, and the bivariate stochastic process $\{Z_\mR(t),  Z_\mI(t) \}_{t\in S}$ is also stationary. The assumption of Gaussianity on $Z$ implies that the finite dimensional distributions are uniquely determined through the second order properties of $Z$. In particular, we define the covariance function $\gamma: S \to \mathbb{C}$, through
\[
   \gamma(\tau) = \E \left\{ Z(t+\tau) Z^*(t)\right\}, \qquad t,\tau \in S,
\] 
where $*$ stands for the transpose conjugate operator. Covariance functions are positive definite: for any finite system of complex constants $(c_k )_{k=1,\dots,N} \subset \mathbb{C}$, and points $t_1,\ldots,t_N$ of $S$, we have $ \sum_{j,k} c_j \gamma(t_j-t_k) c^{*}_k \ge 0. $ We analogously define the cross covariances $\gamma_{\mR,\mI}$ and $\gamma_{\mI, \mR}$, as 
$ \gamma_{\mR,\mI}(\tau)= \E \left\{ Z_{\mR}(t+\tau) Z_{\mI}^*(t) \right\}$, for $t,\tau \in S$, and $ \gamma_{\mI,\mR}(\tau)= \E \left\{ Z_{\mI}(t+\tau) Z_{\mR}^*(t) \right\}$. A relevant remark is that  $\gamma_{\mR,\mI}$ and $\gamma_{\mI, \mR}$ are not, in general, positive definite. 
Instead, the matrix-valued mapping
\[ \left( \begin{array}{cc}
\gamma_{\mR}(\tau) & \gamma_{\mR,\mI}(\tau) \\
\gamma_{\mI,\mR}(\tau) & \gamma_{\mI}(\tau)  
 \end{array} \right)\] 
with $\gamma_{j} \equiv \gamma_{jj}$, $j={\mR,\mI}$, is positive definite according to previous definition, and it is precisely the covariance mapping associated to the stochastic process $\{Z_\mR(t),  Z_\mI(t) \}_{t\in S}$. The following identity is true:
\[
  \gamma(\tau) = \gamma_\mR(\tau) + \gamma_\mI(\tau) + \i \left\{ \gamma_{\mR\mI}(\tau)-\gamma_{\mI\mR}(\tau)\right\} = \gamma^*(-\tau), \qquad \tau \in S,
\]
where it is useful to note that, for $j,k=\mR$ or $\mI$, $ \gamma_{jk}(\tau) =\gamma_{kj}(-\tau)$. 

The aim of the present paper is to generate a discrete sample path of the process $Z$ at times $j=0,1,\dots, n-1$, that is to generate a complex normal vector $\bZ=\{Z(0),\dots,Z(n-1)\}^\top$ with length $n$, zero mean and with covariance matrix $\mathbf \Gamma=\E (\bZ \bZ^*)$ given by
\begin{equation}\label{eq:Sigma}
\mathbf  \Gamma =\left( 
\begin{array}{lllll}
\gamma(0) & \gamma^*(1) & \dots &\gamma^*(n-2) &\gamma^*(n-1) \\
\gamma(1) & \gamma(0) &  \gamma^*(1) & \vdots & \gamma^*(n-2) \\
\vdots & \ddots & \ddots & \ddots & \vdots \\
\gamma(n-2) & \gamma(n-3) & \dots &\gamma(0) & \gamma^*(1)\\
\gamma(n-1) & \gamma(n-2) & \dots & \gamma(1) & \gamma(0)
\end{array}
  \right).
\end{equation}

The covariance function $\gamma$ does not   determine  uniquely the properties of a complex-valued Gaussian process. This can be achieved if, in addition to $\gamma$, the complementary autocovariance function (also called the relation or pseudo--covariance function) $h:S \to \R$,   defined through $h(\tau)= \E \left\{ Z(t+\tau) Z(t)\right\}$, is given (see e.g. \citet[Chapter 8]{lee:messerschmitt:94}), the class of circularly-symmetric processes being an exception.  A complex-valued process is said to be circularly-symmetric if $h(\tau)=0$ for any $\tau \in S$, in which case a stationary complex-valued Gaussian process is uniquely determined by its covariance function. Elementary calculations show that for circularly-symmetric stationary processes
\[
  \gamma_\mR(\tau)=\gamma_\mI(\tau) \quad \mbox{ and }\quad \gamma_{\mR\mI}(\tau)=-\gamma_{\mI\mR}(\tau)= -\gamma_{\mR\mI}(-\tau), \qquad \tau \in S.
\]
For a given class of pseudo-covariances, we can define the matrix $\mathbf H=\E (\bZ \bZ^\top)$. In this paper, we propose an algorithm for generating a complex-valued Gaussian vector with prescribed covariance matrix $\mathbf \Gamma$. We do not focus on the matrix $\mathbf H$, which will be controlled a posteriori, the class of circularly-symmetric processes being again a notable exception. For example, The Cholesky decomposition method decomposes $\mathbf\Gamma$ as $\mathbf L\mathbf L^*$ where $\mathbf L$ is a lower triangular matrix and sets $\bZ =\mathbf L \mathbf N_n$ where $\mathbf N_n$ is a centered complex Gaussian vector with identity covariance matrix. In particular, we can check that if $\bN_n$ is real, $\E (\bN_n\bN_n^*)=\E (\bN_n\bN_n^\top)=\mathbf I_n$ and  the covariance and  relation matrices are respectively equal to $\mathbf\Gamma$ and $\mathbf H = \mathbf L\mathbf L^\top$, that is $\bZ \sim CN(0,\mathbf\Gamma,\mathbf H)$, with $CN$ meaning complex normal. If $\mathbf N_n$ is a circular centered complex normal random vector with identitiy covariance matrix, then $\bZ \sim CN(0,\mathbf \Gamma,\mathbf 0)$.



To generate a complex normal vector $\bZ$ with covariance matrix $\mathbf \Gamma$ and pseudo-covariance matrix $\mathbf H$, from a complex stationary process $Z$, one can simulate the bivariate Gaussian vector $(\bZ_{\mR}, \bZ_{\mI})$ from the bivariate stationary process $(Z_\mR,Z_\mI)$, and set $\bZ=\bZ_{\mR}+\i \bZ_{\mI}$. The simulation of multivariate Gaussian time series is considered by \citet{chan:wood:99} and has been nicely revisited by \citet{helgason:pipiras:abry:11}. We did not consider this direction in this paper as we aimed to exploit the complex characteristic of the process $Z$. Doing this, our algorithm, except for circularly-symmetric processes, does not control beforehand the pseudo-covariance, but its computational cost is clearly smaller than the one required to generate a bivariate Gaussian time series. Moreover, the algorithms proposed by~\citet{chan:wood:99} and~\citet{helgason:pipiras:abry:11} obviously require the covariance functions $\gamma_\mR$ and $\gamma_\mI$, as well as the cross-covariance functions $\gamma_{\mR\mI}$ and $\gamma_{\mI\mR}$, to be given. Instead, the method described in the next section will only assume the complex covariance function $\gamma$ to be given. Such a construction seems to be more natural especially for circularly-symmetric Gaussian processes.


\begin{example}[Modulated stationary process] \label{ex1}

\def\e{\mathrm{e}}

Let $r: S \to \R$ be the covariance function of a real-valued, Gaussian, and stationary stochastic process $\{ Y(t)\}_{t\in S}$. Let $Z$ be the complex-valued Gaussian process defined as $Z(t)= \e^{2 \i \pi t} Y(t) $, $t \in S$ and $\phi \in \R$. Then, straightforward calculations show that 
\begin{equation} \label{jesper}
\gamma(\tau)=e^{2\i\pi \phi \tau}r(\tau), \qquad \tau \in S
\end{equation}
is the covariance function of $Z$, which is called a modulated Gaussian process. 
Similar constructions can then be implemented using the fact that covariance functions are a convex cone being closed under the topology  of finite measures. For instance, for a collection of $p$ uncorrelated real-valued Gaussian processes $Y_k$ with covariance $r_k$, the complex-valued Gaussian process, defined through $Z(t) =\sum_k^p \e^{2 \i \pi \phi_k t} Y_k(t)$, $t \in S$, for $\phi_k \in \R$ for all $k=1,\ldots,p$, has covariance function 
\begin{equation}
  \label{eq:sumMod} \gamma(\tau ) = \sum_{j=1}^p e^{2\i\pi\phi_j \tau} r_j(\tau), \qquad \tau \in S
\end{equation}
Remarkably, for such a construction, the range of dependence, defined as the lag beyond which becomes negligible, is the maximum of the ranges related to each of the covariances $r_k$. Let us list a few examples from this construction:
\begin{itemize}

\item Exponential modulated process: let $p=1$,  
under the exponential model $r(\tau)= \sigma^2 \e^{-\alpha \tau}$, where $\sigma^2$ is the variance and $0<\alpha$ a scaling parameter, the related construction leads to
\begin{equation}\label{eq:exponential}
\gamma(\tau) = \sigma^2 e^{-\alpha \tau + 2\i \pi \phi \tau} , \,\, \qquad \tau \in S.
 \end{equation}
\item Complex autoregressive process of order 1: this process is defined by the equation
\[
  Z(t)- a Z(t-1) = \varepsilon(t), \; t\in \Z
\]
where $a \in \mathbb C$ such that $|a|<1$ and $\{\varepsilon(t)\}_{t \in \Z}$ is a complex normal white noise with variance $\sigma^2$. Then, $Z$ is a stationary process and its covariance function is given for any $\tau$ by $\gamma(\tau)= a^{|\tau|} \sigma^2 (1-|a|^2)^{-1}$. If the white noise is circularly-symmetric then so is $Z$. Finally, letting $a=\rho e^{2\i \pi\phi}$, we note that $\gamma(\tau) = e^{2\i\pi \phi \tau} \rho^{|\tau|} \sigma^2 (1-|\rho|) = e^{2\i\pi \phi \tau} r(\tau)$ where $r$ is the covariance function of a stationary real-valued AR(1) process.

\item \citet{percival:06} considered for example the sum of two modulated covariance functions of FARIMA processes
\[
  \gamma(\tau) = \sum_{k=1}^2 \ e^{2\i \pi \phi_k \tau } r(\tau; d_k), \qquad \tau \in S
\]
where $r(\cdot;d)$ is the autocovariance function of a FARIMA process with fractional difference parameter $d\in [-1/2,1/2]$, and with innovations variance $\sigma_\varepsilon^2$ given for $\tau\in \N$, by, see e.g. \citet{brockwell:davis:87},
\begin{equation}
  r(\tau;d ) = \sigma_\varepsilon^2 \frac{ (-1)^\tau \Gamma(1-2d)}{\Gamma(1-d+\tau)\Gamma(1-d-\tau)}. \label{eq:farima}
\end{equation}
\end{itemize}
On the basis of this construction, when $p=1$, a realization of $\bZ$ can be simply obtained as follows: generate two independent realizations $\bY_1$ and $\bY_2$ of $Y$ at times $0,1,\dots,n-1$ using, for instance, the embedding circulant matrix method for real-valued stationary Gaussian processes \citep{wood:chan:94}. Then, set $(\bZ_j)_\ell=e^{\i \phi \ell } (\bY_j)_\ell$, $j=1,2$. Finally, obtain the realization $\bZ$ through the identity $\bZ=\bZ_1+\i \bZ_2$. The latter has the desired covariance matrix and is ensured to be circular.
When $p>1$ this strategy can still be extended but is more computationally intensive and less natural than directly simulate a circular complex normal vector with the right covariance function.
\end{example}

\begin{example}[Complex fractional Brownian motion] \label{ex3} We define the complex fractional Brownian motion as the self-similar Gaussian process $\tilde Z$, equal to zero at zero, with stationary increments. The self-similarity property is  understood as 
\begin{equation} \label{eq:defSS}
  \tilde Z(\lambda t) \stackrel{fidi}{=} \lambda^H \tilde Z(t) 
  \quad \Longleftrightarrow \quad \tilde Z_{j}(\lambda t) \stackrel{fidi}{=} \lambda^H \tilde Z_{j}(t), \;\; j=\mR,\mI
\end{equation}
where $t\in \R$, $H\in (0,1)$ is called the Hurst exponent, $\lambda$ is any non-negative real number, the sign $\stackrel{fidi}{=}$ means equality in distribution for all finite-dimensional margins and $\tilde Z_\mR(t)$ (resp. $\tilde Z_\mI(t)$) is the real part (resp. imaginary part) of $\tilde Z(t)$. The self-similarity property~\eqref{eq:defSS} is equivalent to $\{ \tilde Z_{\mathcal R}(\lambda t),\tilde Z_{\mathcal I}(\lambda t)\}=\lambda^H \{ \tilde Z_{\mathcal R}(t),Z_{\mathcal I}(t)\}$, a model called the multivariate fractional Brownian motion, a particular case of operator fractional Brownian motion \citep{didier:pipiras:11}, and studied by \citet{amblard:coeurjolly:lavancier:philippe:13,amblard:coeurjolly:achard:13}. As a direct consequence of these works, the increments process, denoted by $Z=\{Z(t)\}_{t\in \R}$, defined by $Z(t)=\tilde Z(t+1)-\tilde Z(t)$ and referred to as the complex fractional Gaussian noise has covariance function $\gamma$ parameterized, when $H\neq 1/2$, as
\begin{equation}\label{eq:gammaCFBM}
  \gamma(\tau) = \frac12\left\{ \sigma_\mR^2+\sigma_\mI^2 -  2\i \, \eta\,\sigma_\mR \sigma_\mI \,\mathrm{sign}(\tau)  \right\} \left( |\tau-1|^{2H} -2|\tau|^{2H} + |\tau+1|^{2H}  \right) 
\end{equation}
where $\sigma_\mR=\E \{Z_\mR(1)\}^{1/2}$ and $\sigma_\mI=\E \{Z_\mI(1)\}^{1/2}$ are non-negative real numbers and $\eta \in \R$. When $H=1/2$ another parameterization occurs and for the ease of the presentation, we avoid this case. \citet[Proposition~9]{amblard:coeurjolly:lavancier:philippe:13} states that the covariance function~\eqref{eq:gammaCFBM} is a valid covariance function if and only if $\eta^2 \leq  \tan(\pi H)^2$. 

When the process is time-reverisble, i.e. $\tilde Z(t)\stackrel{d}{=}\tilde Z(-t)$ for any $t\in \R$ then, as outlined by~\citet{amblard:coeurjolly:lavancier:philippe:13}, the parameter $\eta$ must be equal to zero, which makes  the covariance function $\gamma$  real. This is not of special interest for this paper. Finally, when $\sigma_\mR=\sigma_\mI=\sigma$, the covariance function reduces to
\begin{equation}\label{eq:CSFBM}
\gamma(\tau) = \sigma^2 \left\{1  -  \i \, \eta\,\mathrm{sign}(\tau)  \right\} \left( |\tau-1|^{2H} -2|\tau|^{2H} + |\tau+1|^{2H}  \right) 
\end{equation}
and it can be checked that the corresponding stochastic process $Z$ is circularly-symmetric. 
\end{example}

\section{Simulation through circulant matrix method}
\label{sec:simulation}

This section deals with circulant embedding method for complex-valued covariance functions. The procedure is an extension of the standard method proposed by \citet{wood:chan:94} for real covariance functions. It is also slightly different from the extension proposed by~\citet{percival:06} to handle complex covariance functions. Then, we discuss the main question of this method which is the non-negativeness of the circulant matrix in which $\mathbf\Gamma$ is embedded. 

\subsection{Simulation of a complex normal vector with  covariance matrix $\mathbf \Gamma$} \label{sec:methodology}

In order to achieve a realization from the Gaussian process $Z$, under the covariance function $\gamma$, at times $0,1, \ldots, n-1$, we need to obtain a realization from $\bZ$  being complex normal, with covariance matrix $\mathbf \Gamma$ given by~\eqref{eq:Sigma}.

Let $m \geq n-1$, $\mt=2m+1$ and let $\bC$ be the $\mt \times \mt$ circulant matrix defined by its first row $ \{ c_j, j=0,\dots,2m \}$, where
\begin{equation}
  \label{eq:Cm}
   c_j =\left\{
\begin{array}{ll}
  \gamma(0) & \mbox{ if } j=0\\
  \gamma^*(j) & \mbox{ if } j=1,\dots,m \\
  \gamma(2m+1-j) & \mbox{ if } j=m+1,\dots,2m. 
  \end{array}
   \right.\end{equation}
By construction, the top left corner of $\bC$ corresponds to the covariance matrix $\mathbf \Gamma$. 
Standard results for symmetric circulant matrices, see \citet{brockwell:davis:87}, show that the Hermitian matrix $\bC$ can be decomposed as $\bC = \bQ \bL \bQ^*$, where $\bL=\mathrm{diag} \{\lambda_0,\dots,\lambda_{\mt-1}\}$ is the diagonal matrix of real eigenvalues of $\bC$, $\bQ$ is the matrix with entries
\begin{equation}\label{eq:defQ}
  (\bQ)_{jk}= \mt^{-1/2} e^{- \frac{2\i \pi jk}{\mt}}, \qquad j,k=0,\dots,\mt-1.
\end{equation}
If $\bC$ is non-negative, that is if $\lambda_k\geq0$ for $k=0,\dots,\mt-1$, the simulation method consists simply in picking the first $n$  components of the vector $\bQ \bL^{1/2}\bQ^*\mathbf N_\mt$ where $\mathbf N_\mt$ is a complex normal vector with mean 0 and identity covariance matrix. The main advantage being the fast computation of the eigenvalues, additionally with minimal storage when using Fast Fourier Transform (FFT). 

The procedure proposed in this paper is similar to the algorithm proposed by \citet{wood:chan:94}, who worked in the real-valued case, and by~\citet{percival:06} in the complex case. In particular, our first algoritmh extends \citet{wood:chan:94} and considers $\bN_\mt$ as a real vector, that is a vector of $\mt$ independent standard Gaussian random variables. The second algorithm, presented in Section~\ref{sec:simulationcircular} considers $\bN_\mt$ as a circular complex normal vector with identity covariance matrix.\\

\noindent{\sc Algorithm 1.}\\
\noindent{\it Step 0.} Let $m\geq n-1$, be an odd number (preferably a highly composite number). Embed the matrix $\mathbf \Gamma$ into the circulant matrix $\bC$ given by~\eqref{eq:Cm}. 

\medskip

\noindent{\it Step 1.} Determination of the eigenvalues $\lambda_0,\dots,\lambda_{\mt-1}$. The calculation of $\bL = \bQ^* \bC \bQ$ leads to
\begin{equation}
  \label{eq:lambdak}
  \lambda_k = \sum_{j=0}^{\mt-1} c_j e^{- \frac{2\i \pi jk}{\mt}}, \quad k=0,\dots,\mt-1.
\end{equation}
Check that all eigenvalues are non-negative (Section~\ref{sec:positivedefiniteness} provides some conditions on $\gamma$ which ensure this fact). If some of them are negative, increase $m$ and go back to Step 0 or set the negative eigenvalues to 0. With the latter option, discussed in more details in Section~\ref{sec:approximation}, the simulation will be only approximate.

\medskip

\noindent{\it Step 2.} Simulation of $\mathbf W=\{W_0,\dots,W_{\mt-1}\}^\top = {\mt}^{-1/2}\bL^{1/2} \bQ^* \bN_\mt$. This is achieved using the following result.

\begin{proposition}\label{prop:step2} For $k=0,\dots,\mt-1$
\[
  (\mathbf W)_k= W_k = \, \sqrt{\frac{\lambda_k}{2\mt}} \times \left\{
\begin{array}{ll}
  S_k+\i T_k  & \mbox{ for } k=0,\dots,m \\
  S_{\mt-k}-\i T_{\mt-k}  & \mbox{ for } k=m+1,\dots,\mt-1,
\end{array}
  \right.
\]
in distribution, where for $k=0,\dots,m$, $S_k$ and $T_k$ are real--valued Gaussian random variables with mean 0 and variance 1, and $S_0,\dots,S_m,T_0,\dots,T_m$ are mutually independent.  
\end{proposition}

\medskip
\noindent{\it Step 3.} Reconstruction of $\bZ$. This step results in calculating $\bQ\mathbf\Lambda^{1/2}\bQ^* \bN_\mt = \mt^{1/2} \bQ \mathbf W$ and keep the first $n$ components, which corresponds to the calculation of
\begin{equation}
  \label{eq:step3}
  (\bZ)_k=  \sum_{j=0}^{\mt-1} W_j e^{- \frac{2\i \pi jk}{\mt}}, \quad k=0,\dots,n-1.
\end{equation}

Step 2  requires  the simulation  $2m+2$ independent realizations of standard Gaussian random variables. This is computationally less expensive than the similar step of the algorithm proposed by \citet{percival:06}, which,  with the notation of the present paper, requires $4m$ realizations of Gaussian variables. Since $\mt=2m+1$ is an odd number, the proof of Step 2 is also slightly different from \citet[Proposition 3.3]{wood:chan:94}. Steps 1 and 3 can be handled very quickly using the direct FFT. \\

\medskip

Some comments are in order. In the real-valued case, $\mathbf \Gamma$ is real and symmetric by construction. In particular, we have 
 $\gamma(m)=\gamma^*(m)$, so that the dimension of $\bC$ can be reduced to $2m\times 2m$, where $m$ is an integer being larger than $2(n-1)$, and $2m$ can be set to a power of two. In the complex-valued case, $\bC$ has dimension $(2m+1)\times(2m+1)$, with $(2m+1)$ being necessarily an odd number.

\citet{percival:06} used a specific modulation of the initial process $Z$ to force $\gamma(m)$ to be real, that is instead of generating $\bZ$ with covariance matrix $\mathbf \Gamma$, the idea is to generate $\check\bZ= \left\{e^{\i \nu k} (\bZ)_k\right\}_{k=1,\dots,n}$ where $\nu$ is chosen such that $e^{i \nu m} \gamma^*(m)$ is a real number. This modulation enables to recover a circulant matrix $\bC$ with dimension $2m\times 2m$, $2m$ can still be set to a power of two, which allows the use of 'powers of two'  FFT algorithm for diagonalizing $\bC$. Forcing $\gamma(m)$ to be real has however some minor drawbacks: first, if we increase the value of $m$, the modulation changes and the first row of $\bC$ is completely modified. Second, the introduction of the modulation modifies the covariance function $\gamma$. The resulting covariance function is less easy to handle from a theoretical point of view, in particular when we want to provide  conditions on $\gamma$ ensuring $\bC$ to be non-negative.\\
Defining $\bC$ by~\eqref{eq:Cm} imposes the number of rows to be an odd number. However, this is not of great importance because FFT algorithm (like the one implemented in the \texttt{R} function \texttt{fft}) is very efficient when $2m+1$ is highly composite, that is has many factors, see \citet{brockwell:davis:87} or \citet{nag:93}. We explore this in Table~\ref{tab:fft}. Remind that $n$ is the length of the desired sample path of $Z$. Using a specific modulation of $Z$, \citet{percival:06} suggested to use a minimal embedding which corresponds to a circulant matrix whose first row length,  denoted by $\mt_{\mathrm{mod}}$, is the first power of $2$ larger $2(n-1)$. When $\bC$ is defined by~\eqref{eq:Cm}, we let $\mt$ be the first power of $3,5,7,11$ or a combination of these powers larger than $2n-1$. Table~\ref{tab:fft} reports average time in milliseconds of FFT algorithm applied to vector of length equal to $\mt_{\mathrm{mod}}$ or $\mt$ for different values of $n$. For the values of $n$ considered in Table~\ref{tab:fft}, we can always find a highly composite integer number $\mt <\mt_{\mathrm{mod}}$. As a consequence of this, we observe a time reduction when a FFT is applied whereby we conclude that there is no reason to focus on embedding into a circulant matrix with first row as a power of two. Therefore, we did not consider the modulation suggested by \citet{percival:06}.

\begin{table}[ht]
\centering
\begin{tabular}{rrrrrrrr}
  \hline
 & n=1000 & 5000 & 10000 & 50000 & $100000$ & $500000$ & $1000000$ \\
  \hline
$\tilde m_{\mathrm{mod}}=2^p$ & 0 & 1 & 2 & 13 & 25 & 210 & 470 \\
  $\tilde m=3^p$ & 0 & 1 & 3 & 10 & 82 & 331 & 1258 \\
  $5^p$ & 0 & 1 & 4 & 47 & 47 & 369 & 2553 \\
  $7^p$ & 0 & 1 & 7 & 8 & 136 & 1319 & 1343 \\
  $11^p$ & 1 & 1 & 10 & 10 & 296 & 282 & 5106 \\
  $3^{p_1}5^{p_2}$ & 0 & 0 & 2 & 6 & 15 & 198 & 410 \\
  $3^{p_1}5^{p_2} 7^{p_3}$  & 1 & 1 & 2 & 6 & 12 & 173 & 439 \\
  $3^{p_1}5^{p_2} 7^{p_3}11^{p_4}$ & 0 & 1 & 1 & 6 & 12 & 169 & 383 \\
   \hline
\end{tabular}
\caption{Average time (in ms) of FFT applied to vectors (obtained as realizations of standard Gaussian random variables) of length $\tilde m_\mathrm{mod}$ or $\tilde m$. Ten replications are considered. We restrict attention on the cases $p_3\geq 1$ and $p_1\wedge p_2>0$ for the second to last row and on the cases $p_4\geq 1$ and  $p_1\wedge p_2\wedge p_3>0$ for the last row. Experiments are performed on a $1.7$ GHz Intel Core i7 processor.
\label{tab:fft}}
\end{table}

\medskip

{\sc Algoritm 1}  does not control the relation matrix $\mathbf H$ but we can have an idea of its form. This is given by the following result.

\begin{proposition}\label{prop:H}
Let $\bZ_\mt$ be the output vector of {\sc Algoritm 1}, then the relation matrix $\mathbf H$ of $\bZ=(\bZ_\mt)_{0:(n-1)}$ corresponds to the top left corner of
 $\mathbf H_\mt= \E \bZ_\mt \bZ_\mt^\top = \bQ^* \mathbf V \bQ$ where $\mathbf V=\mathrm{diag}(v_k, k=0,\dots,\mt-1)$ is the diagonal matrix with elements given by $v_0=0$ and $v_k= \sqrt{\lambda_k \lambda_{\mt-k}}$ for $k\geq 1$, where $\lambda_k$, $k=0,\dots,\mt-1$ are the eigenvalues of $\bC$ given by~\eqref{eq:lambdak}. Thus, $\bQ^* \mathbf V \bQ$ is necessarily a circulant matrix and $\mathbf H$ is necessarily a Toeplitz Hermitian matrix with first row given by 
 \[
   \mathbf H_{0k} = \sum_{j=1}^{\mt-1} \sqrt{\lambda_j \lambda_{\mt-j}} e^{-\frac{2\i \pi jk }{\mt}}.
 \]
\end{proposition}

\subsection{Simulation of a circular complex normal with  covariance matrix $\mathbf \Gamma$} \label{sec:simulationcircular}

This section focusses on the circularly symmetric case, for which $\mathbf H=\mathbf 0$. A realization $\bZ$ from such a process can be obtained as follows: let $\bZ_{1}$ and $\bZ_{2}$ be two output vectors from {\sc Algorithm 1}. Then, set $\bZ = (\bZ_1+ \i \bZ_2)/\sqrt{2}$. 
This in turn results in a modification of {\sc Algorithm 1}: in Step 2,  $\bN_\mt$ is replaced by a circular complex normal random vector, i.e. the vector $(\bN_{1,\mt}+\i \bN_{2,\mt})/\sqrt 2$ where $\bN_{1,\mt}$ and $\bN_{2,\mt}$ are two real-valued, mutually independent, random vectors of independent standard Gaussian random variables.\\

\noindent{\sc Algorithm 2.}\\
\noindent{\it Steps 0 and 1.} Similar to Steps 0 and 1 of {\sc Algorithm~1}.
\medskip

\noindent{\it Step 2.} Simulation of $\mathbf W=\{W_0,\dots,W_{\mt-1}\}^\top = {\mt}^{-1/2}\bL^{1/2} \bQ^* (\bN_{1,\mt}+\i \bN_{2,\mt} )/\sqrt 2$. This is achieved using the following result.

\begin{proposition}\label{prop:step2algorithm2} For $k=0,\dots,\mt-1$, 
\[
  W_k = \sqrt{\frac{\lambda_k}{2\mt} }\left ( S_k+\i T_k  \right)
\]
in distribution, where for $k=0,\dots,m$, $S_k$ and $T_k$ are  Gaussian random variables with mean 0 and variance 1, and $S_0,\dots,S_m,T_0,\dots,T_m$ are mutually independent.  
\end{proposition}

\medskip
\noindent{\it Step 3.} Similar to Step 3 of {\sc Algorithm 1}.

\medskip

It is worth noticing that Step 2 of {\sc Algorithm 2} now requires the simulation of $4m+2$ realizations of standard Gaussian distributions and is very similar to the corresponding step of the algorithm proposed by~\citet{percival:06}. 

We think the distinctions we make between the two algorithms we propose, make more clear the understanding of the consequences of each algorithm on the relation matrix $\mathbf H$.

\subsection{Approximation and error control} \label{sec:approximation}

In this section, we focus on circularly-symmetric processes and propose a modification of {\sc Algorithm~2} when $\bC$ is negative. When, it is practically unfeasible to increase the value of $m$ and reperform Steps~0 and~1, we follow \citet{wood:chan:94} and suggest to truncate the eigenvalues to 0. The simulation becomes only approximate.

This section details the procedure and provides a control of the approximation error. 
We decompose $\bC$ as follows
\[
  \bC = \bQ \Lambda \bQ^* = \bQ (\bL_+ - \bL_-) \bQ^* = \bC_+ - \bC_-
\]
where $\bL_\pm = \mathrm{diag} \{\max(0,\pm \lambda_k), \; k=0,\dots,\mt-1 \}$. We suggest to replace $\bC$ in Step~1 by $\varphi^2 \bC_+$ with $\varphi = \mathrm{tr}(\bL)/\mathrm{tr}(\bL_+)$. Let $\bZ^\app$ be the output vector of {\sc Algorithm~2}, which is a circular centered complex normal vector with covariance matrix $\mathbf \Sigma^\app$ equal to the top left corner of $\big(\varphi^2 \bC_+\big)$. It is worth noticing that this choice for $\varphi$ leads to $(\varphi^2 \bC_+)_{jj}=(\mathbf \Sigma)_{jj}$, for $j=0,\dots,n-1$. Let $\bZ$ be a complex normal vector independent of $\bZ^\app$, with zero mean and covariance matrix $\mathbf \Sigma$. We define $\mathbf \Delta= \bZ-\bZ^\app$ as the random error of approximation. Clearly, $\mathbf \Delta$ is a  circular centered complex normal vector with covariance matrix $\mathbf \Sigma-\mathbf \Sigma^\app$. Using multivariate normal probabilities on rectangles proposed by~\citet{dunn:58,dunn:59} (see also \citet[chapter~2]{tong:82}), we obtain the following approximation.

\begin{proposition} \label{prop:approximation}
  Let $s_j^2=\Var \Delta_j$, $s_{j,\mR}^2=\Var \mathrm{Re}(\Delta_j)$ and $s_{j,\mI}^2=\Var \mathrm{Im}(\Delta_j)$, for $j=0,\dots,n-1$. Then, for each $x>0$
  \begin{equation}
  \label{eq:app}
  \PP\left(\max_{j=0,\dots,n-1} \sigma_j^{-1} |\Delta_j| > x\right)  \leq
  1 - \prod_{j=0}^{n-1} \prod_{k \in {\cal T}}
  \left\{ 2 \Phi\left(  \frac{x\sigma_j}{\sigma_{j,{\cal M}}\sqrt{2}}\right) -1\right\},
  \end{equation}
  for ${\cal T}=\{ \mR,  \mI\}$. 
\end{proposition} 

For different values of $x$, \eqref{eq:app} can be used to control the largest normalized error.

\section{Non-negativeness of $\bC$} \label{sec:positivedefiniteness}

The condition ensuring the circulant matrix $\bC$ to be non-negative is now discussed. \citet{dietrich:newsam:97} and~\citet{craigmile:03}
dealt with the real-valued case, and combination of their results covers many interesting classes of covariance functions. We now show how to extend these results to the complex-valued case.

Let us first express the eigenvalues $\lambda_k$ in terms of the covariance function $\gamma$. By~\eqref{eq:Cm} and~\eqref{eq:lambdak}, we have for any $k=0,\dots,\mt-1$
\begin{align}
\lambda_k &= \gamma(0) + \sum_{j=1}^m \gamma^*(j)e^{-\frac{2\i \pi j k}{\mt}} + \sum_{j=m+1}^{2m} \gamma(\mt-j) e^{-\frac{2\i \pi j k}{\mt}} \nonumber\\
&= \gamma(0) + \sum_{j=1}^m \left\{
  \gamma^*(j)e^{-\frac{2\i \pi j k}{\mt}} + \gamma(j) e^{\frac{2\i \pi j k}{\mt}}
\right\}\nonumber\\
&= \gamma(0) + 2\sum_{j=1}^m \left\{ 
\mR(j) \cos\left(\frac{2\pi jk}{\mt}\right) -
\mI(j) \sin\left(\frac{2\pi jk}{\mt}\right) 
\right\}, \label{eq:lambdak2}
\end{align}
where $\mR$ and $\mI$ correspond to the real and imaginary parts of the complex covariance function $\gamma$, that is $\mR(j) = \gamma_\mR(j) + \gamma_\mI(j)$ and $\mI(j) = \gamma_{\mR\mI}(j)-\gamma_{\mI\mR}(j)$.

The first result extends the main result of~\citet{craigmile:03}.

\begin{proposition}\label{prop:craigmile} For any integer $m\geq n-1$, let $\lambda_k$, $k=0,\ldots,\mt-1$ as defined through Equation (\ref{eq:lambdak2}). Then, either of the following conditions are sufficient for $\lambda_k$ to be non-negative for all $k$. \\
\noindent(i) For $j\in \Z\setminus\{0\}$, $\mR(j)$ is negative and $s=\sign \{j\mI(j)\}$ is constant. Additionally, for any $j\in \Z$, the matrix 
\begin{equation}\label{eq:M}
  \mathbf M(j) = \left(
\begin{array}{ll}
\gamma_\mR(j) & \sign(j) \gamma_{\mR\mI}(j) \\
-\sign(j) \gamma_{\mI\mR}(j) & \gamma_{\mI}(j) 
\end{array}
  \right)  
\end{equation}
is the covariance matrix of a bivariate stationary process on $\Z$ which admits a well--defined spectral density matrix $\mathbf S$. \\
(ii) For a circularly-symmetric stationary process $Z$ such that $\gamma_{\mR\mI}(j) = \eta \sign(j) \gamma_\mR(j)$ for $j\neq 0$ for some parameter $\eta \in [-1,1]$, and such that $\gamma_\mR(j)$ is negative for $j\geq 1$.\\
(iii) The covariance function $\gamma$ is defined according to Equation (\ref{jesper}) and $r \ge 0$ on $\mathbb{Z}_+$.
\end{proposition}

The next result extends \citet[Theorem~2]{dietrich:newsam:97}. For a sequence $(f_k)_{k \ge 0}$ of real numbers, we denote the first and second order finite differences by $\Delta f_k = f_k-f_{k+1}$ and $\Delta^2 f_k=\Delta f_k-\Delta f_{k+1}$. The sequence $(f_0,\dots,f_k,\dots)$ is said to be decreasing and convex respectively, if $\Delta f_k\geq 0$ and $\Delta^2 f_k \geq 0$ for $k\geq0$. 

\begin{proposition} \label{prop:dietrich} Let the functions $\widetilde D$, $K$ and $\widetilde K$ be, respectively, the conjugate Dirichlet, the Féjer and conjugate Féjer kernels, as being defined in Lemma~\ref{lem:zygmund} (Appendix). For any $k=0,\dots,\mt-1$, let $\lambda_k$ as being defined through Equation \eqref{eq:lambdak2}. Then, it is true that 
\begin{align}
  \lambda_k &= -\mI(m) \widetilde D_m\left(\frac{k}\mt\right) + (-\Delta \mI)(m-1) \widetilde K_{m-1} \left(\frac k\mt\right) + \sum_{j=1}^{m-2}  (-\Delta^2 \mI)(j) \widetilde K_j\left(\frac k\mt\right) \nonumber\\
  &\quad + \Delta \mR(m-1) K_{m-1}\left(\frac k\mt\right) + \sum_{j=0}^{m-2}(\Delta^2 \mR)(j) K_j\left( \frac k\mt\right). \label{eq:lambdakIPP}
\end{align}

Also, for any integer $m\geq n-1$,  the following conditions are sufficient for $\lambda_k$ to be non-negative: \\
(i) The two sequences $\{\mR(0),\dots,\mR(m)\}$ and $\{-\mI(1),\dots,-\mI(m)\}$  are both decreasing and convex, $-\mI(m)\geq 0$ and
\begin{equation}
  \label{eq:conddietrich}
  \Delta^2 \mR(0)  +S_m \geq  -\mI(m)
\end{equation}
where \[
S_m=\inf_{k=0,\dots,\mt-1}\sum_{j=1}^{m-2} 
\left\{
\Delta^2 \mR(j) K_j\left(\frac{k}\mt\right)-\Delta^2 \mI(j) \widetilde K_j\left(\frac k\mt \right)  \right\}.
\]
(ii) For $Z$ a circularly-symmetric stationary process such that \linebreak$\gamma_{\mR\mI}(j) = -\eta \sign(j) \gamma_\mR(j)$ for $j\neq 0$ and some parameter $\eta>0$, the sequence $\{\gamma_\mR(0),\dots,\gamma_\mR(m)\}$ is decreasing and convex, with $\gamma_\mR(m)\geq0$ and
\begin{equation}
  \label{eq:conddietrich2}
   \Delta^2\gamma_{\mR}(0) + S_m(\eta)\geq  \eta\gamma_{\mR}(m) 
\end{equation}
where \[
S_m(\eta)=\inf_{k=0,\dots,\mt-1}\sum_{j=1}^{m-2} 
\Delta^2 \gamma_\mR(j)\left\{ K_j\left(\frac{k}\mt\right) + \eta \widetilde K_j\left(\frac k\mt \right)  \right\}.
\]
\end{proposition}

\begin{remark}
By Lemma~\ref{lem:zygmund}, the Féjer kernel is always non-negative. The conjugate Féjer kernel is also non-negative for any $k$ such that $k/\mt<1/2$. Therefore, the infimum involved in conditions~\eqref{eq:conddietrich} and~\eqref{eq:conddietrich2} can be taken over the set $\{m+1,\dots,\mt-1\}$. This term $S_m$ could look annoying as it seems to depend strongly on $m$. We investigate this in Section~\ref{sec:illustrations} on a specific example.
\end{remark}

\begin{remark}
If in (i), $\{\mI(1),\dots,\mI(m)\}$ is a decreasing and convex sequence, or if in (ii)  $\eta<0$, then instead of simulating $Z$, we simulate $Z^*$: the expresssion of $\lambda_k$ would be in this case
\[
\lambda_k=  \gamma(0) + 2\sum_{j=1}^m \left\{ 
\mR(j) \cos\left(\frac{2\pi jk}{\mt}\right) +
\mI(j) \sin\left(\frac{2\pi jk}{\mt}\right) 
\right\}  
  \]  
  and Proposition~\ref{prop:dietrich} can be applied if in addition~\eqref{eq:conddietrich} or~\eqref{eq:conddietrich2} holds.
\end{remark}

The following result extends~\citet[Theorem~2]{dietrich:newsam:97} for modulated stationary covariances.

\begin{proposition} \label{prop:dietrichMod}
Assume that the covariance function $\gamma$ is a modulated stationary covariance, that is, there exists $\phi \in \R$ and a real-valued covariance function $r$, such that $\gamma(\tau) = e^{2\i\pi \phi \tau}r(\tau)$. Assume that $\{r(0),\dots,r(m)\}$ forms a decreasing and convex sequence, then, for any $m\geq n-1$ and $k=0,\dots,\mt-1$, $\lambda_k\geq 0$.  
\end{proposition}

\begin{remark}\label{rem:sum}
If the covariance function $\gamma$ is the sum of $p$ complex covariance functions, that is $\gamma(\tau)=\sum_{j=1}^p \gamma_j(\tau)$, then  the circulant matrix $\bC$ into which $\mathbf \Gamma$ is embedded is also the sum of  circulant matrices $\bC_j$ into which  $\mathbf \Gamma_j$,  the covariance matrices corresponding to $\gamma_j$, are embedded. Hence, the eigenvalues of $\bC$ can be written as $\lambda_k = \sum_{j=1}^p \lambda_k^{(j)}$ for $k=0,\dots,\mt-1$, where $\lambda^{(j)}_k$ are the eigenvalues of $\bC_j$. As a consequence, if for $j=1,\dots,p$, the covariance function $\gamma_j$ satisfies the conditions of Proposition~\ref{prop:craigmile},~\ref{prop:dietrich} or~\ref{prop:dietrichMod}, then the eigenvalues $\lambda_k$ are non-negative. 
\end{remark}


\section{Applications} \label{sec:examples}

\subsection{Back to Examples} \label{sec:illustrations}

We now show that the results in Section~\ref{sec:positivedefiniteness} can be applied to   Examples~\ref{ex1}-\ref{ex3}. In particular, for such classes the minimal embedding is sufficient to ensure the non-negativeness of $\bC$. The method is therefore exact for these examples.

\subsubsection{Modulated stationary processes}

Proposition~\ref{prop:craigmile} (iii) applies to real covariances functions satisfying the assumptions in Proposition~1 of \citet{craigmile:03}. Typical examples are the covariance functions of a FARIMA process with fractional difference parameter $d\in [-1/2,0)$ (see Equation \eqref{eq:farima}) and the covariance function of a fractional Gaussian noise with Hurst parameter $H\in (0,1/2)$. For these two examples, Proposition~\ref{prop:dietrichMod} completes the result, since it can be checked that the aforementioned covariance functions are decreasing and convex when $d\in (0,1/2]$ for the FARIMA process and when $H\in (1/2,1)$ for the fractional Gaussian noise.

Referring to~\citet{berg:forst:12}, here is a list of other examples with positive, decreasing and convex covariance functions, thus satisfying Proposition~\ref{prop:dietrichMod}: 
(a) the mapping $\tau \mapsto r(\tau;\alpha,\beta)=\sigma^2(1+|\tau|^{\alpha})^{-\beta}$, $\alpha \in (0,1]$, $\beta>0$, $\sigma>0$ and $\tau \in \R$; 
(b) $r(\tau)= \sigma^2(1- |\tau|)_{+}^n$, $n>0$ and $\sigma>0$; 
(c) $r(\tau)= \sigma^2\int_{(0,\infty)} (1-\xi |\tau|)_+ \mu({\rm d} \xi)$, for $\mu$ any positive and bounded measure and $\sigma>0$; 
(d) $r(\tau)=\sigma^2 e^{-\alpha |\tau|}$, $\alpha>0$ and $\sigma>0$; 
(e) $r(\tau)=\sigma^2 \rho^{|\tau|}$, $0<\rho<1$ and $\sigma>0$.

This, in particular, covers the modulated exponential covariance~\eqref{eq:exponential} and the complex AR(1) process presented in Example~\ref{ex1}. Finally, Remark~\ref{rem:sum} can be applied to embrace examples where the covariance is the sum of modulated FARIMA or fractional Gaussian noise covariance functions. Figure~\ref{fig:Mod} illustrates this section.

\begin{figure}[htbp]
\subfigure[Real part of the covariance function]{\includegraphics[scale=.6]{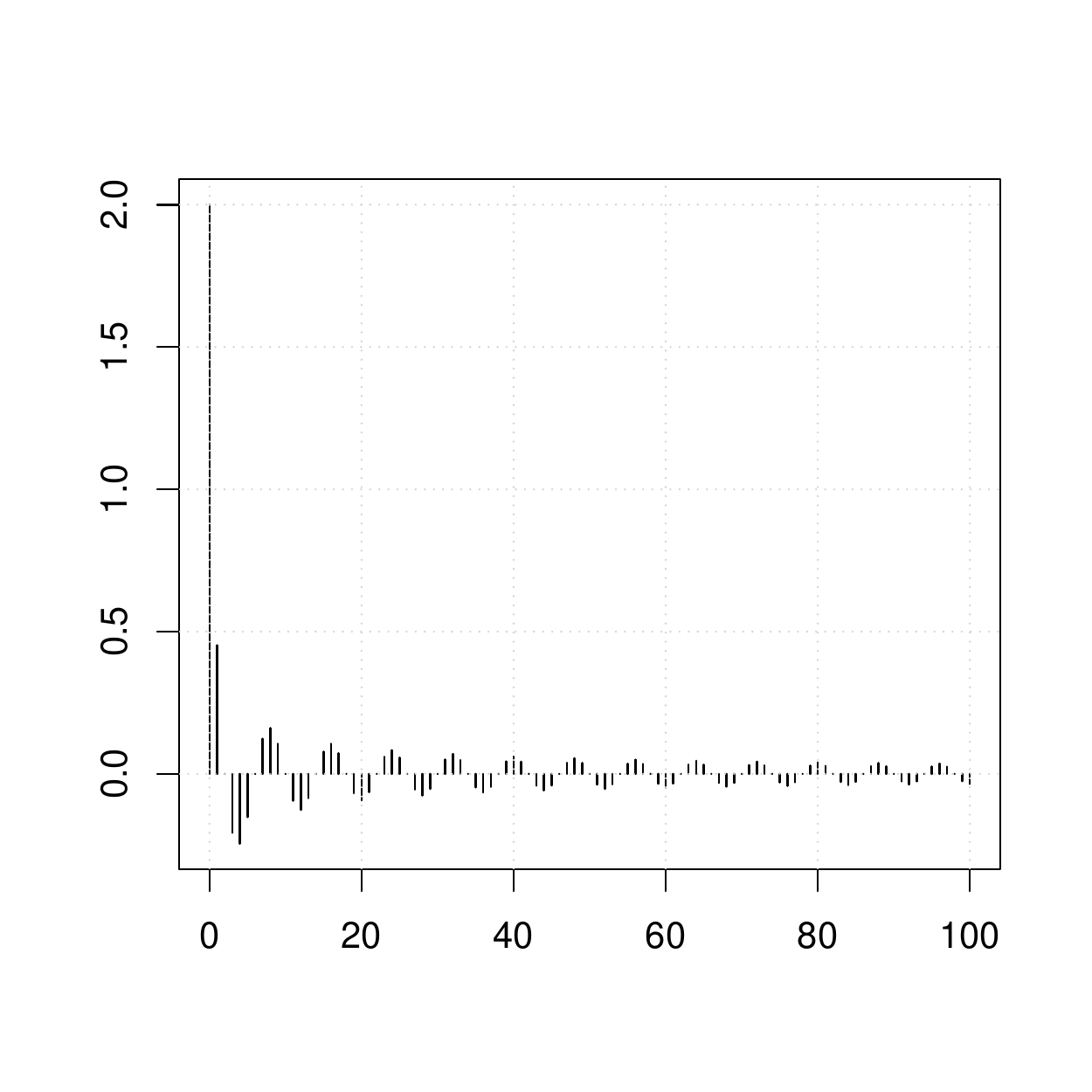}}
\subfigure[Imaginary part of the covariance function]{\includegraphics[scale=.6]{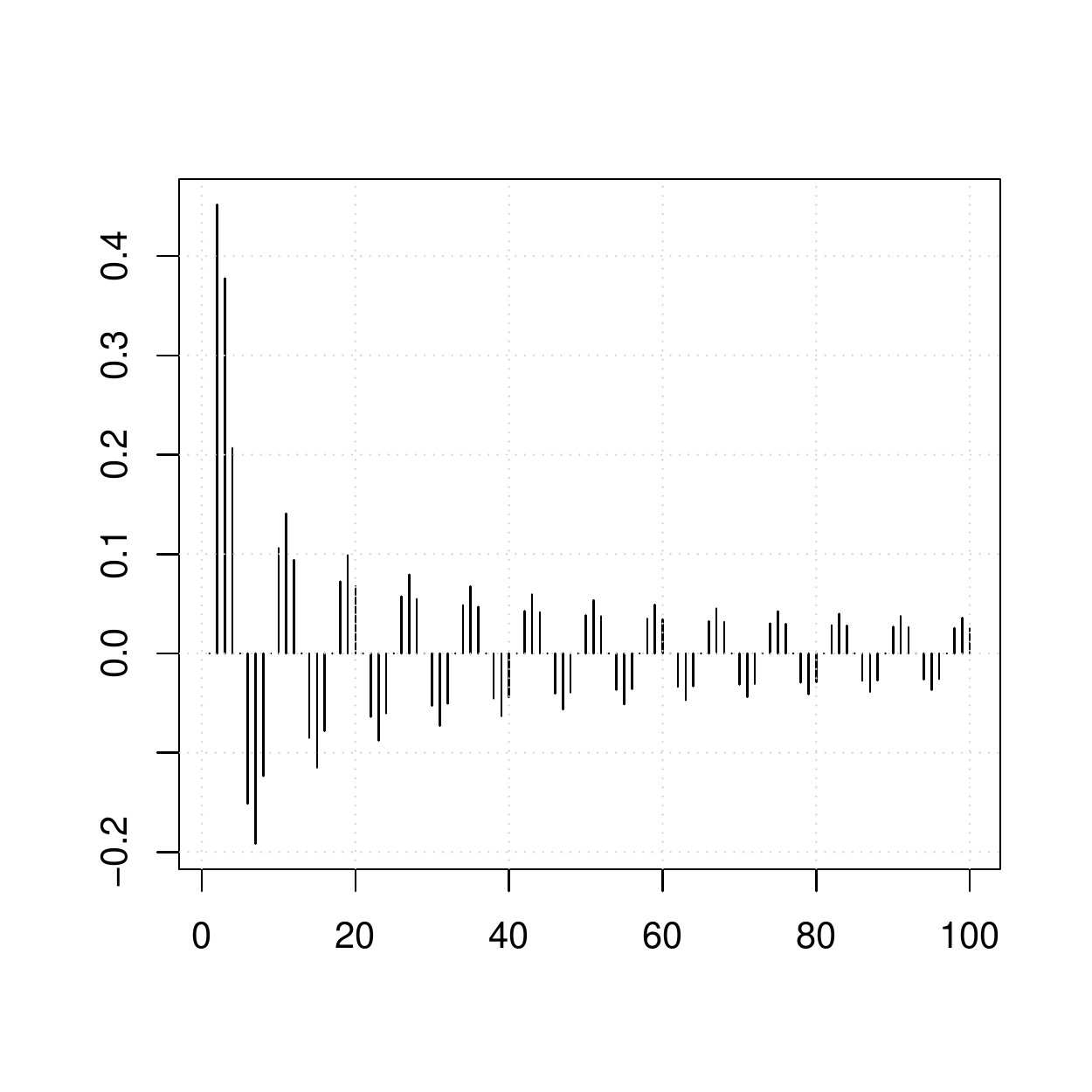}}
\subfigure[Eigenvalues of the circulant matrix $\bC$]{\includegraphics[scale=.6]{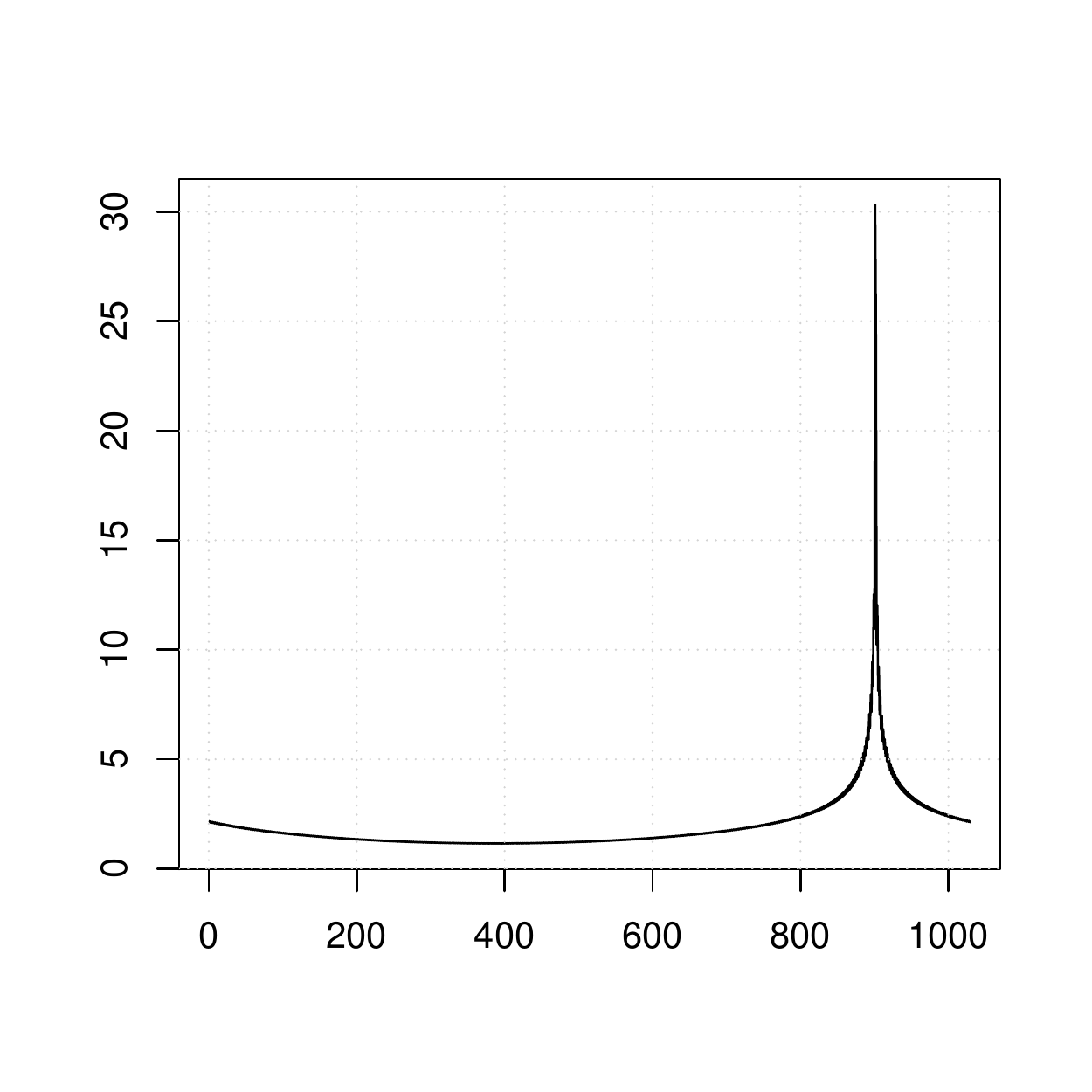}}
\subfigure[Real (top) and Imaginary parts of $Z(t)$]{\includegraphics[scale=.6]{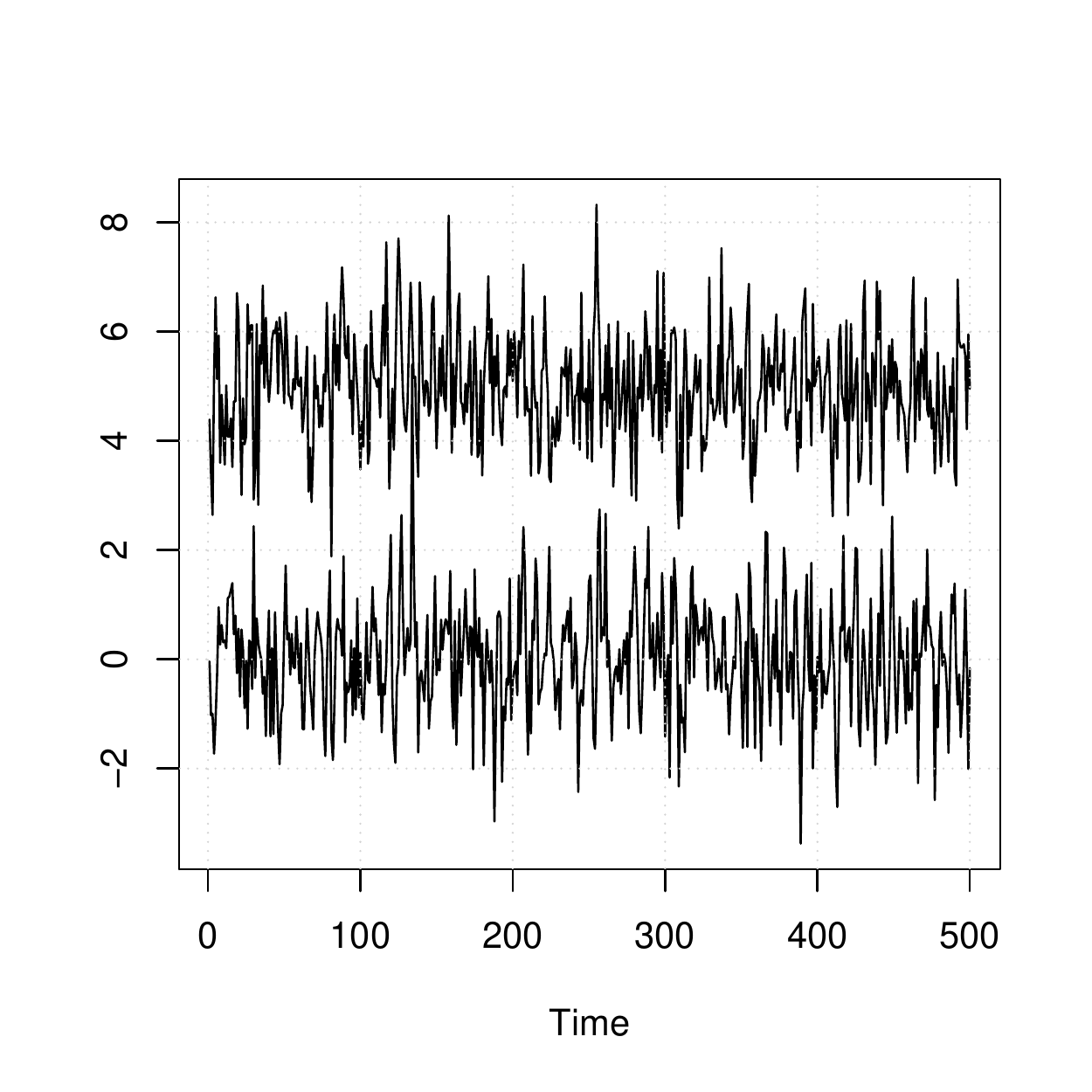}}
  \caption{\label{fig:Mod}Simulation details for the example of a modulated FARIMA$(0,d,0)$ process with unit variance, fractional parameter $d=0.2$ and phase parameter $\phi=1/8$. The sample size is $n=500$ and $\bC$ is chosen as a $m\times m$ matrix with $m=3\times 5\times 7=514$. For (d), a constant is added to the real part of $Z(t)$ to differentiate the two sample paths.}
\end{figure}

\subsubsection{Circular complex fBm}

The circular complex fBm has covariance given by~\eqref{eq:CSFBM}. We omit the case $H=1/2$, which, as outlined earlier, leads to another parametrization of the covariance function. We remind that this covariance function is positive definite under the condition that $\eta\leq |\tan(\pi H|$. We study separately the cases $H\in (0,1/2)$ and $H\in (1/2,1)$. When $H\in (0,1/2)$, Proposition~\ref{prop:craigmile}~(ii) applies with the restriction $\eta< \min\{1,\tan(\pi H)\}$.

When $H\in (1/2,1)$, we apply Proposition~\ref{prop:dietrich} (ii). In this setting,  $\gamma_{\mR}(j)$ corresponds to the covariance function of a fractional Gaussian noise with Hurst exponent $H$. This covariance function is decreasing and convex for any $H\in (1/2,1)$. Let us now comment~\eqref{eq:conddietrich2}. We can establish that the sequence $\{\Delta^2 r(j)\}_{j\geq 1}$ decreases hyperbolically to 0 with a rate of convergence $j^{2H-2}$ and since the Féjer and conjugate Féjer kernels are bounded, it can be expected that $S_m$ is quite small. Added to the fact that $\gamma_\mR(m)\to 0$, we can really expect that~\eqref{eq:conddietrich2} is not restrictive. 

For several values of $m$, we have evaluated  the value of $\tilde H$ such that for the maximal value of the parameter $\eta$ allowed by the model, that is $\eta=|\tan(\pi H)|$, $\Delta^2 \gamma_{\mR}(0)+S_m\{|\tan(\pi H)|\} \geq |\tan(\pi H)|\gamma_{\mR}(m)$ is valid for any $H \in (1/2,\tilde H)$. We obtained the values $\tilde H=0.939, 0.954$ and $0.964$ when $m=100,1000$ and $10000$. The conditions~\eqref{eq:conddietrich}--\eqref{eq:conddietrich2} could be slightly refined, for instance by noticing that $K_{m-1}(k/\mt)\geq 1$. We do not present this since, for instance regarding the value of $\tilde H$ investigated above, we did not notice significant improvements.

Figures~\ref{fig:CFBMH2}, \ref{fig:CFBMH8} and~\ref{fig:circ} illustrate this section. For $H=0.2$ and $H=0.8$. Figures~\ref{fig:CFBMH2}-\ref{fig:CFBMH8} depict the sample paths of the circular complex fBm with length $n=10^6$. Four seconds is the timing required to generate each realization. As expected, we observe that the higher $H$ the more regular the sample path. We can check that despite the plot of the eigenvalues exhibit different shapes, the eigenvalues are all non-negative. Finally, using the \texttt{R} function \texttt{acf}, the circularity property is graphically tested in Figure~\ref{fig:circ}. The difference between the estimates $\gamma_{\mR\mI}$ and $-\gamma_{\mI\mR}(j)$ are very small, which convinces us that the realization should be circular.

\begin{figure}[hbtp]
\subfigure[Real part of the covariance function]{\includegraphics[scale=.6]{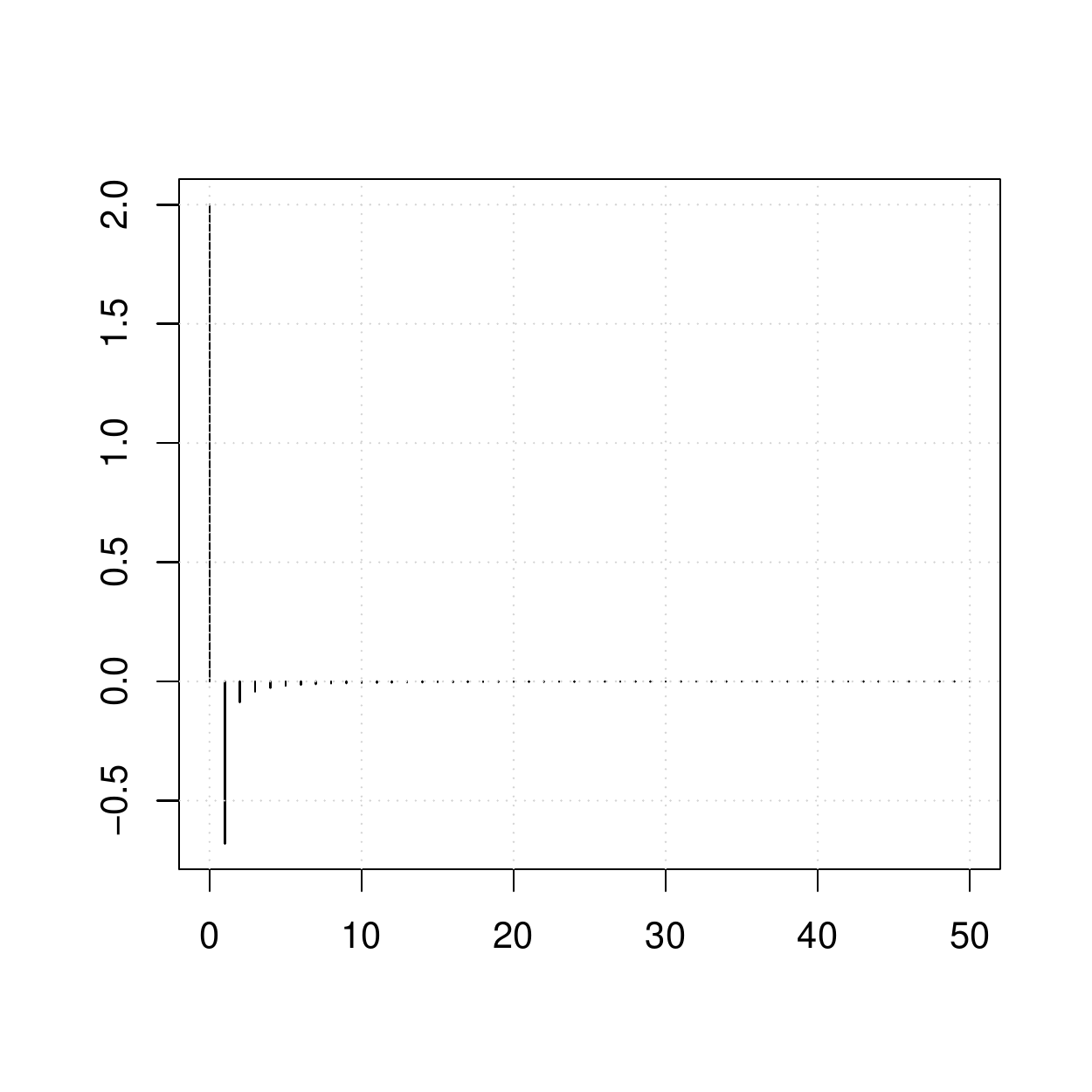}}
\subfigure[Imaginary part of the covariance function]{\includegraphics[scale=.6]{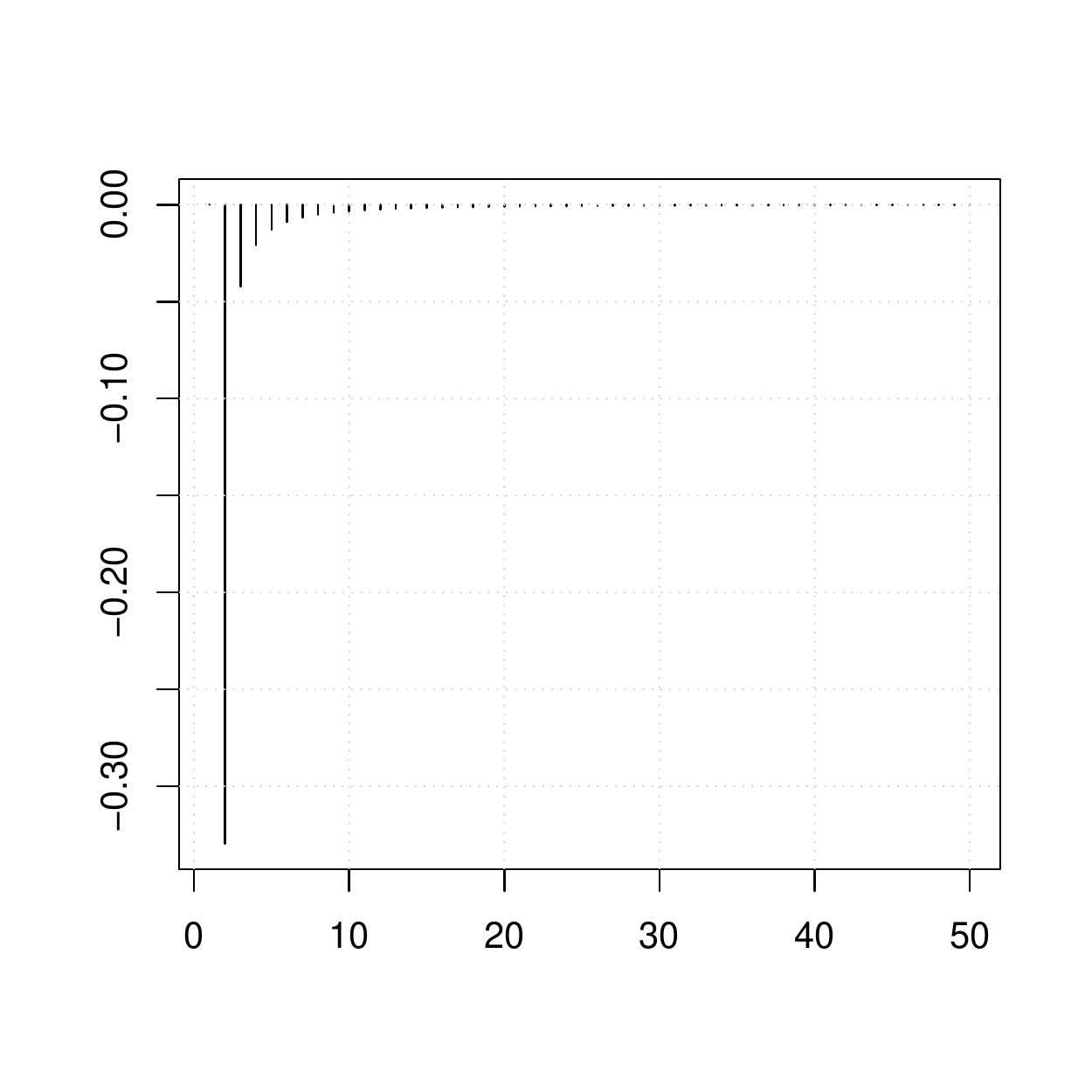}}
\subfigure[Eigenvalues of the circulant matrix $\bC$]{\includegraphics[scale=.6]{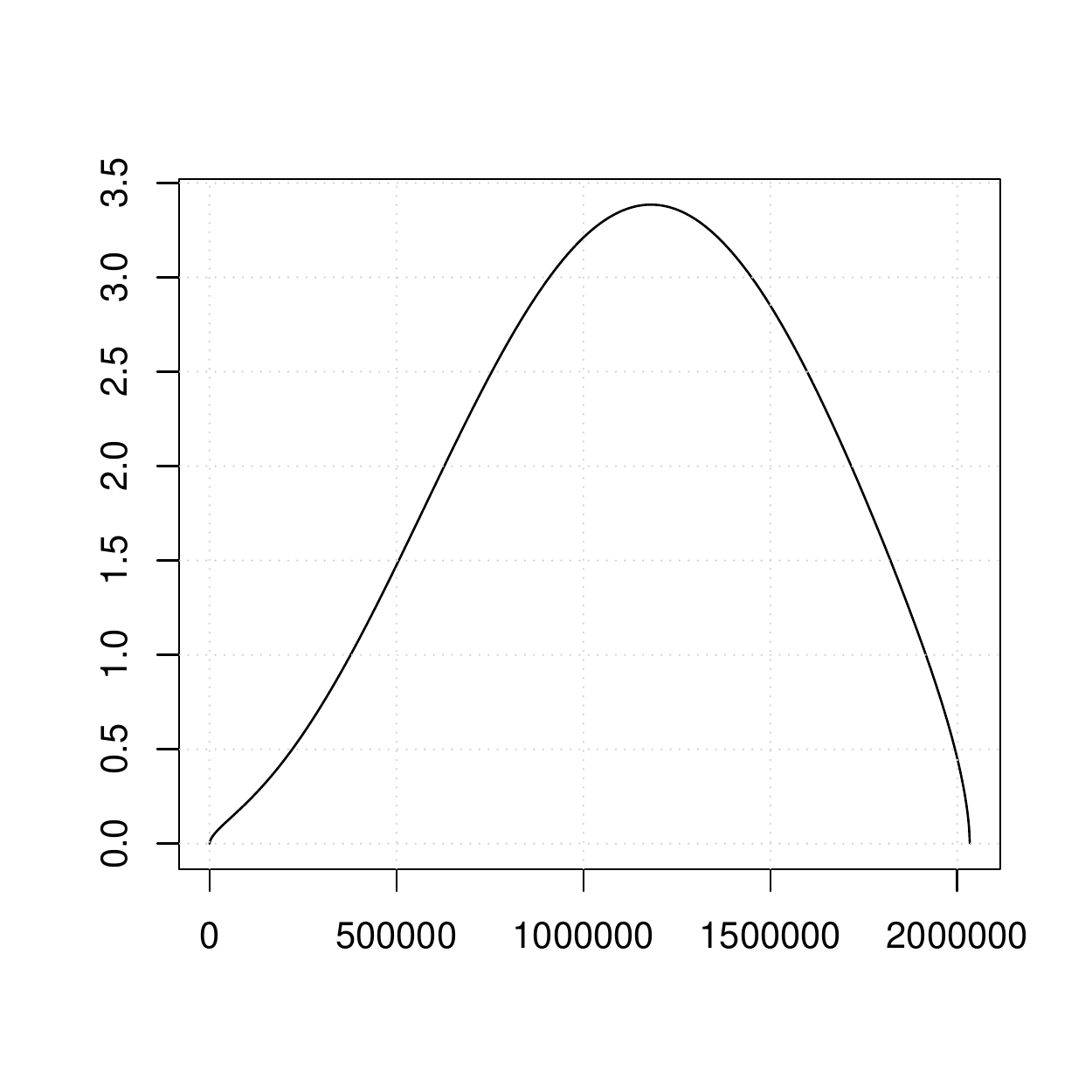}}
\subfigure[Real (top) and Imaginary parts of $\widetilde Z(t)$]{\includegraphics[scale=.6]{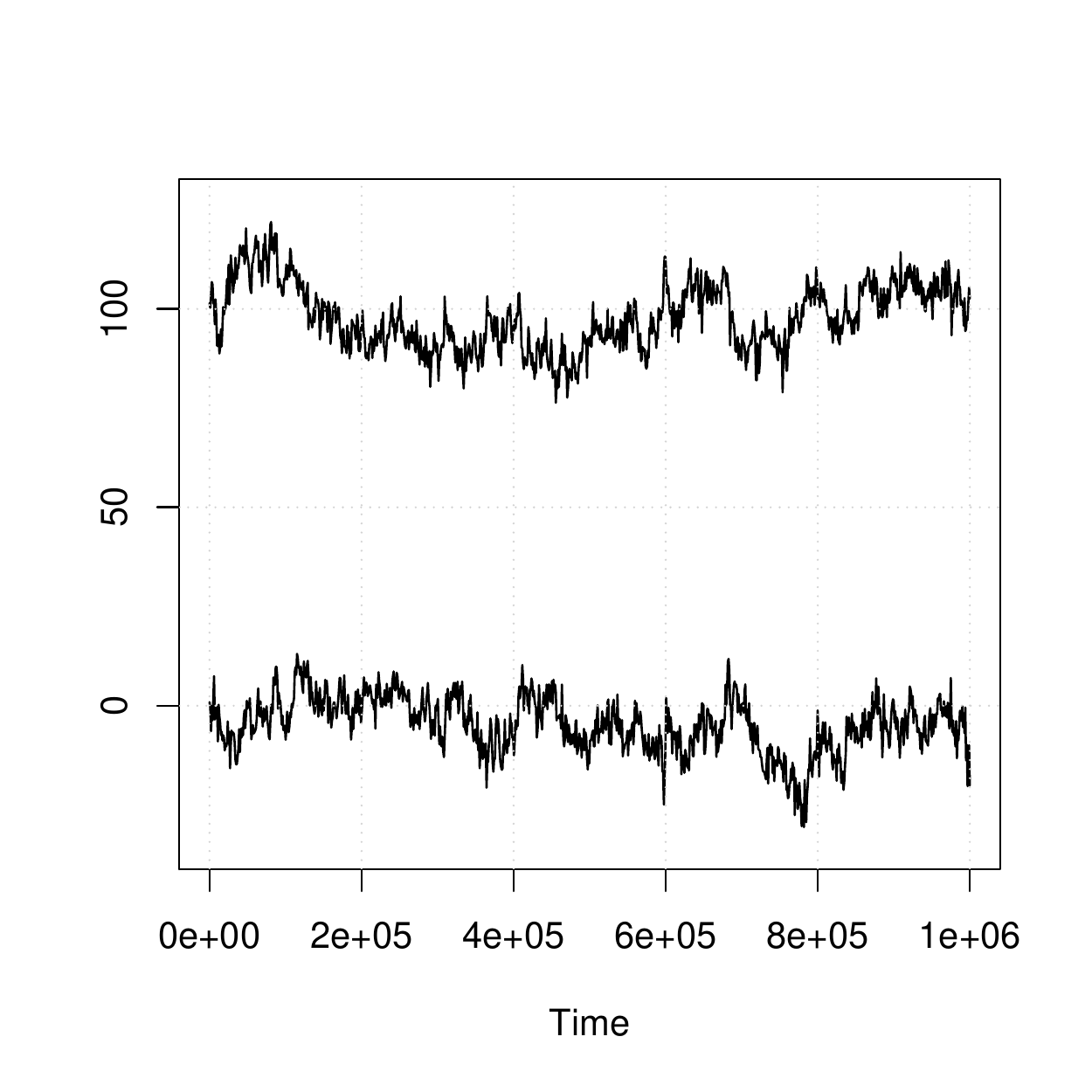}}
  \caption{\label{fig:CFBMH2} 
  Simulation details for the example of a circular complex fBm with unit variance, Hurst exponent $H=0.8$ and $\eta=\frac23 |\tan(\pi H)|$. The sample size is $n=10^6$ and $\bC$ is chosen as a $m\times m$ matrix with $m=2033647$. For (d), a constant is added to the real part of $Z(t)$ to differentiate the two sample paths.}
\end{figure}

\begin{figure}[htbp]
\subfigure[Real part of the covariance function]{\includegraphics[scale=.6]{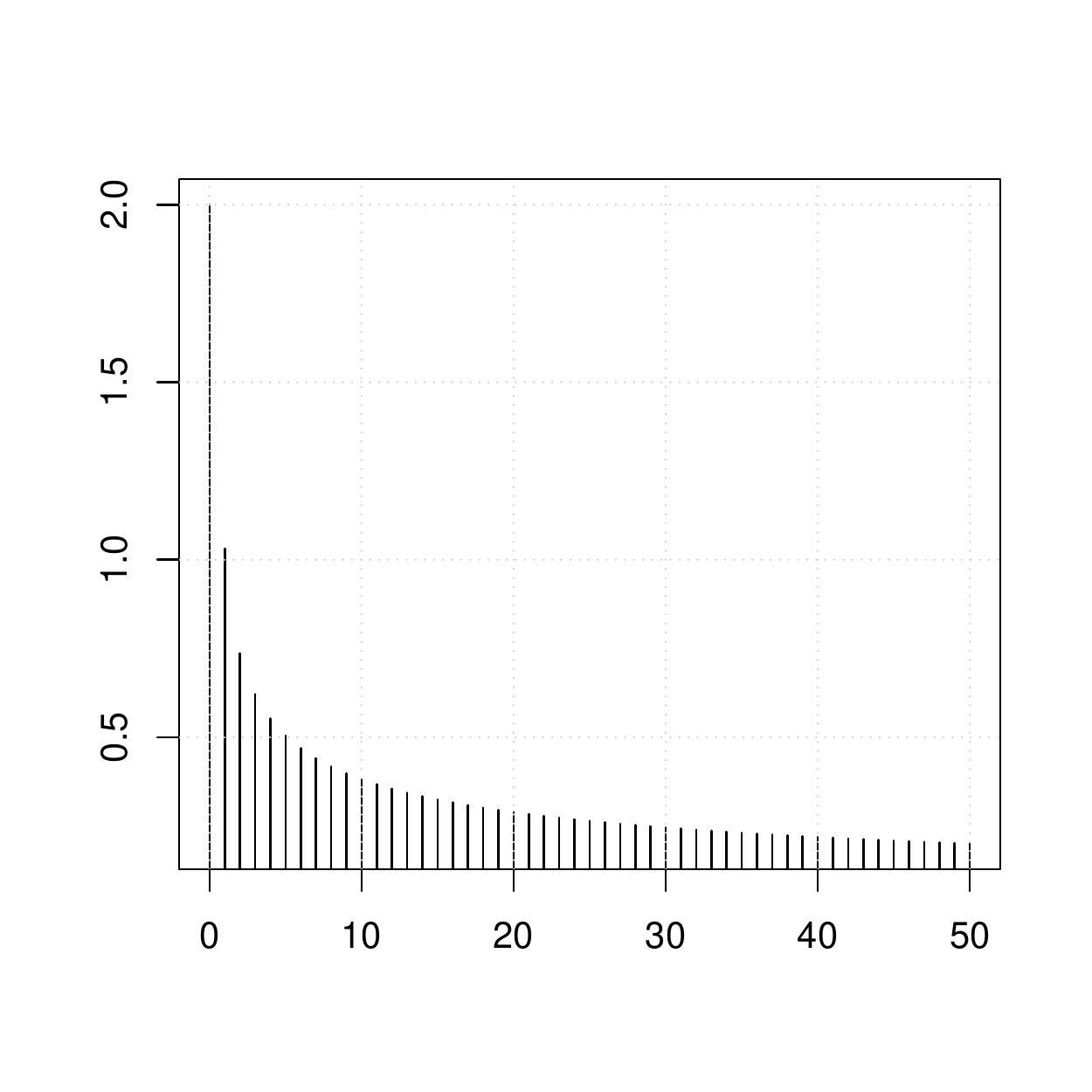}}
\subfigure[Imaginary part of the covariance function]{\includegraphics[scale=.6]{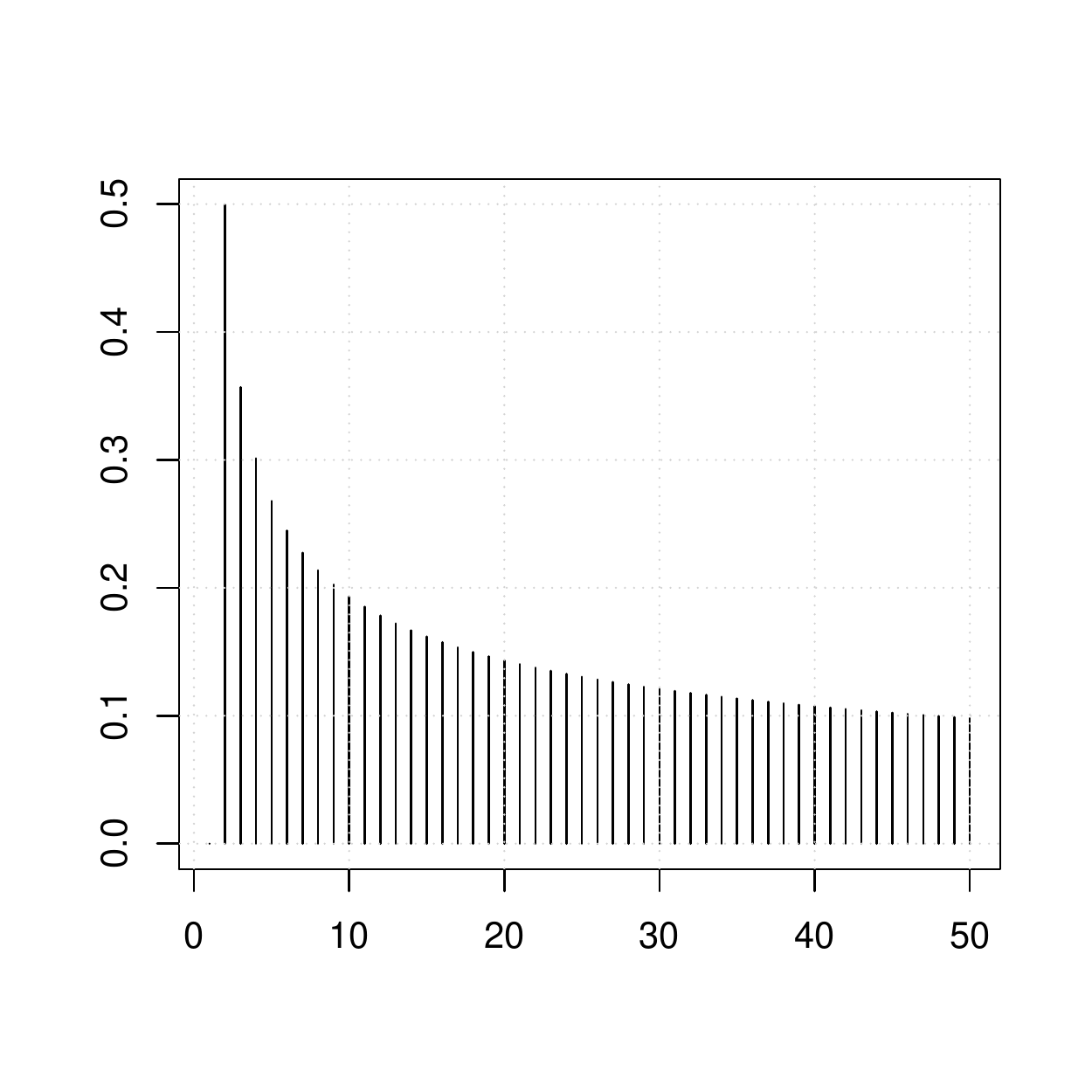}}
\subfigure[Eigenvalues of the circulant matrix $\bC$]{\includegraphics[scale=.6]{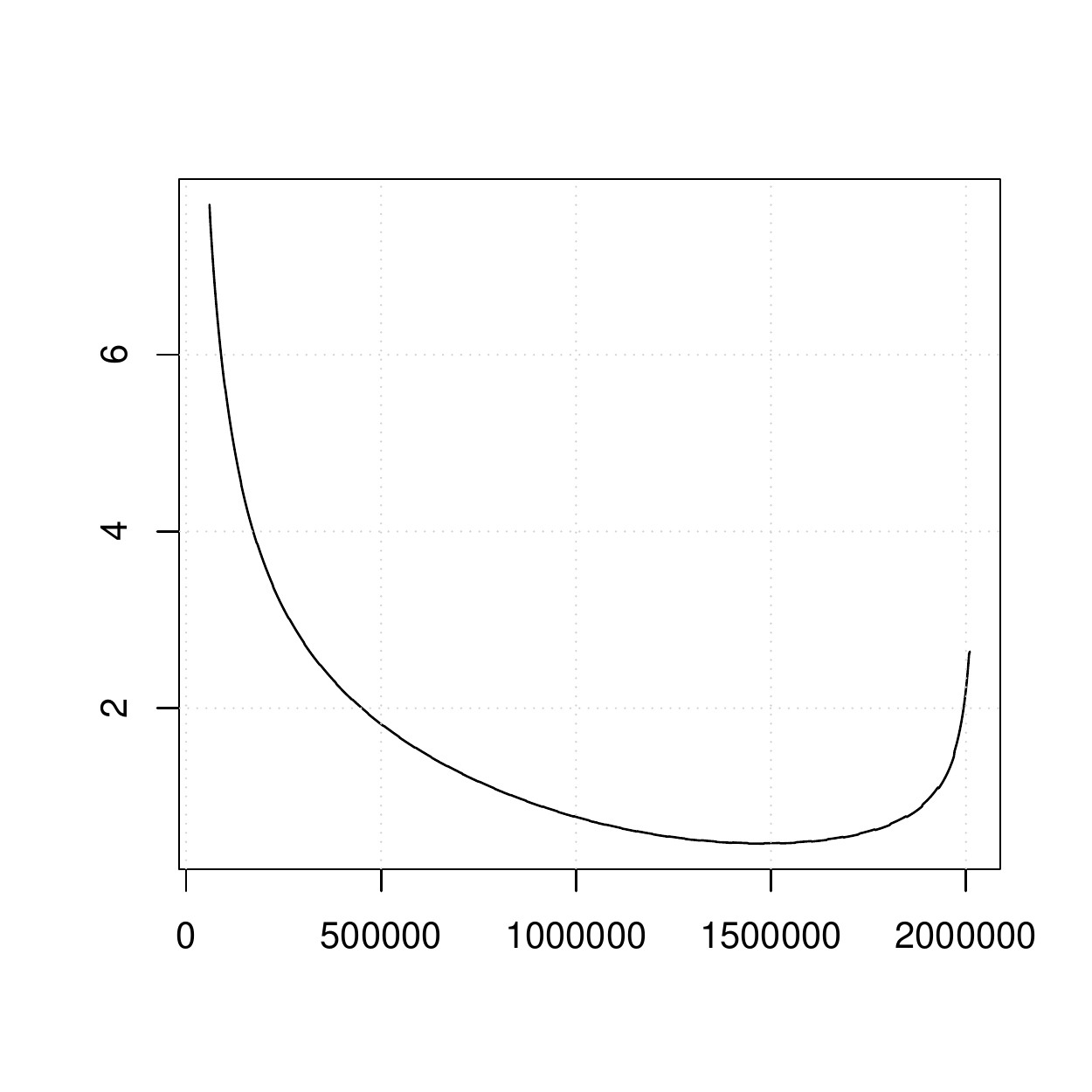}}
\subfigure[Real (top) and Imaginary parts of $\widetilde Z(t)$]{\includegraphics[scale=.6]{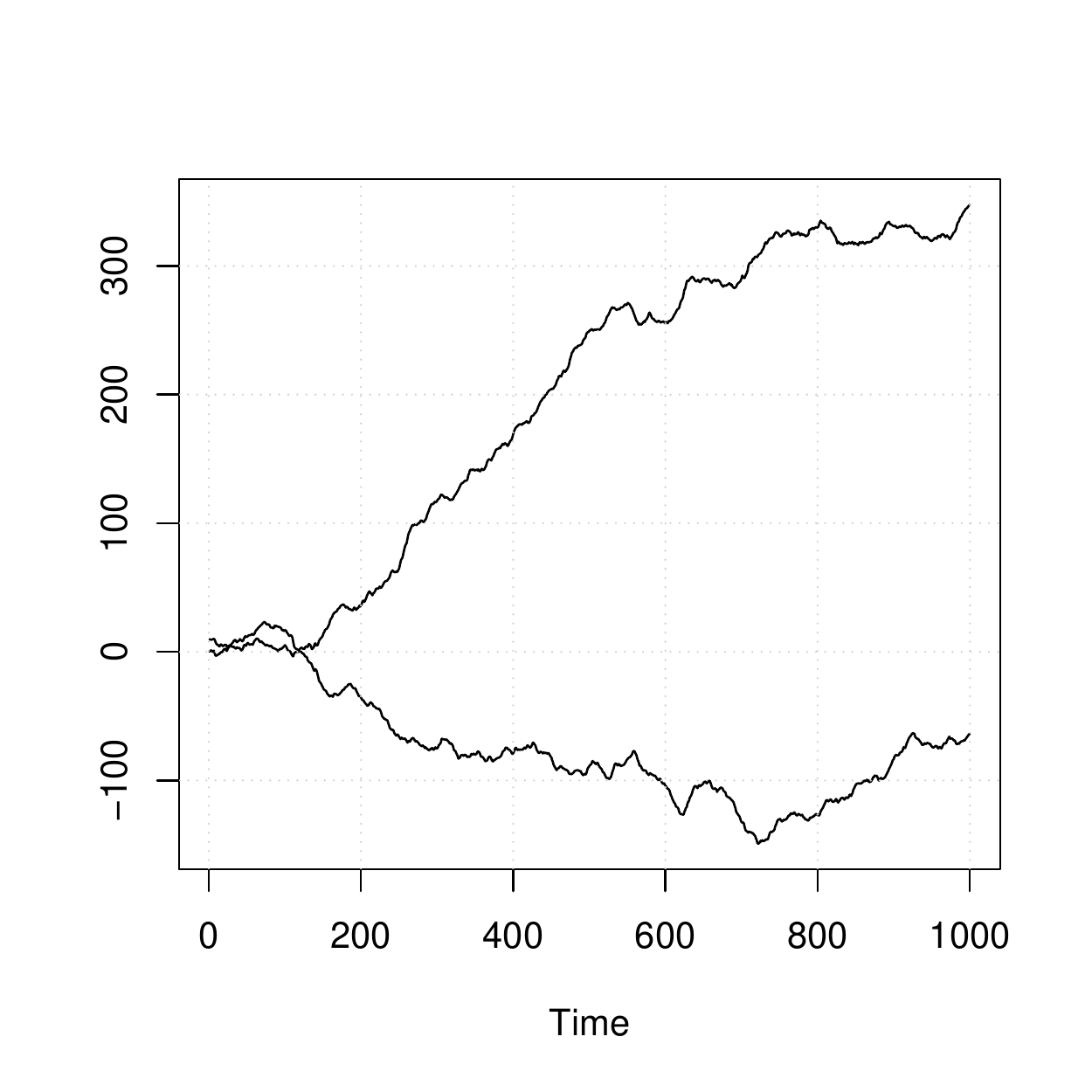}}
\caption{\label{fig:CFBMH8} Simulation details for the example of a circular complex fBm with unit variance, Hurst exponent $H=0.8$ and $\eta=\frac23 |\tan(\pi H)|$. The sample size is $n=10^6$ and $\bC$ is chosen as a $m\times m$ matrix with $m=2033647$. For (c), we focus on the eigenvalues $\lambda_k$ for $k=50000,\dots,2\times 10^6$. The other ones are very large. For (d), a constant is added to the real part of $Z(t)$ to differentiate the two sample paths.}
\end{figure}

\begin{figure}[htbp]
\subfigure[$H=0.2$, $\widehat \gamma_{\mR\mI}(j)$]{\includegraphics[scale=.6]{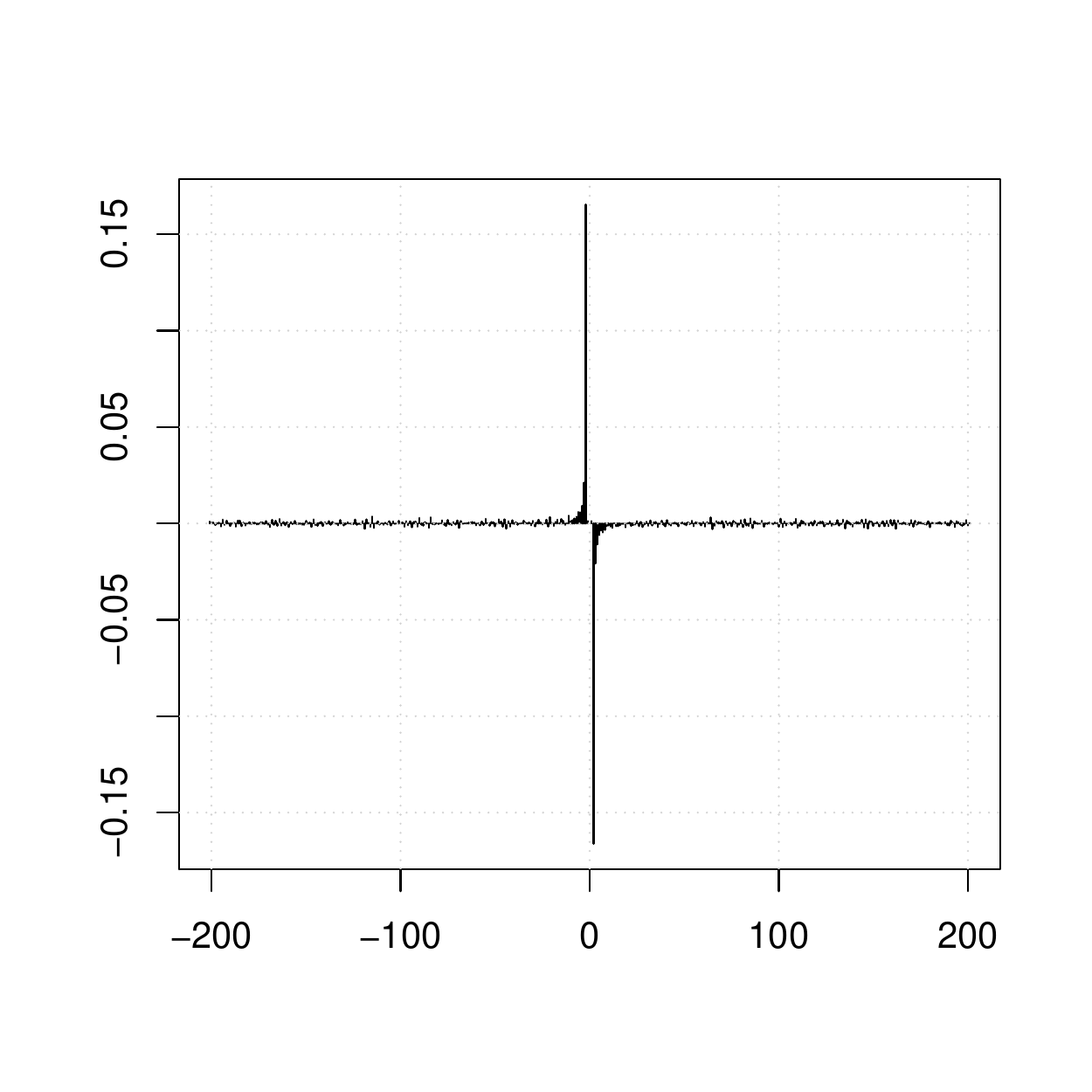}}
\subfigure[$H=0.8$, $\widehat \gamma_{\mR\mI}(j)$]{\includegraphics[scale=.6]{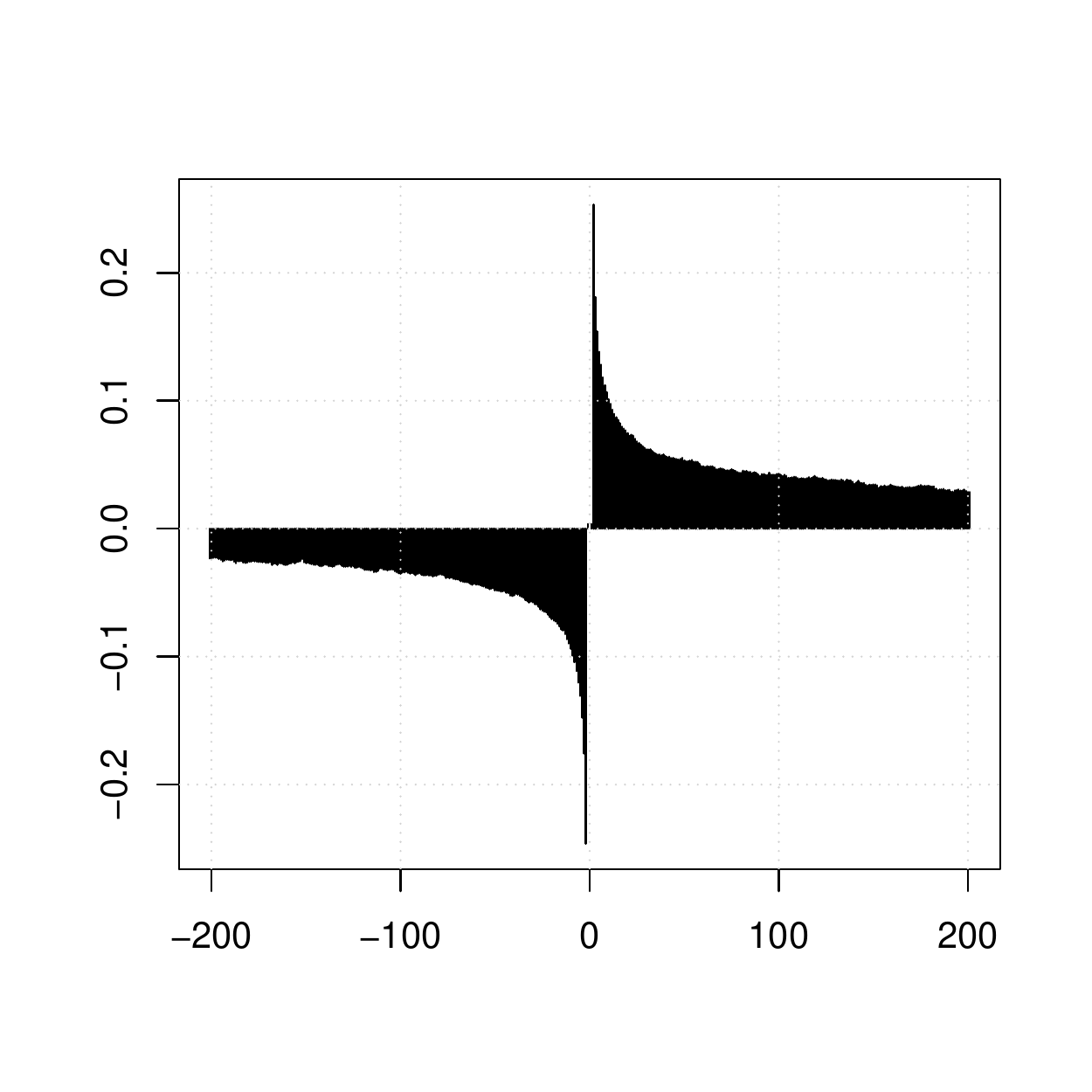}}
\subfigure[$H=0.2$, $\widehat \gamma_{\mR\mI}(j)-\{-\widehat \gamma_{\mI\mR}(j)\}$]{\includegraphics[scale=.6]{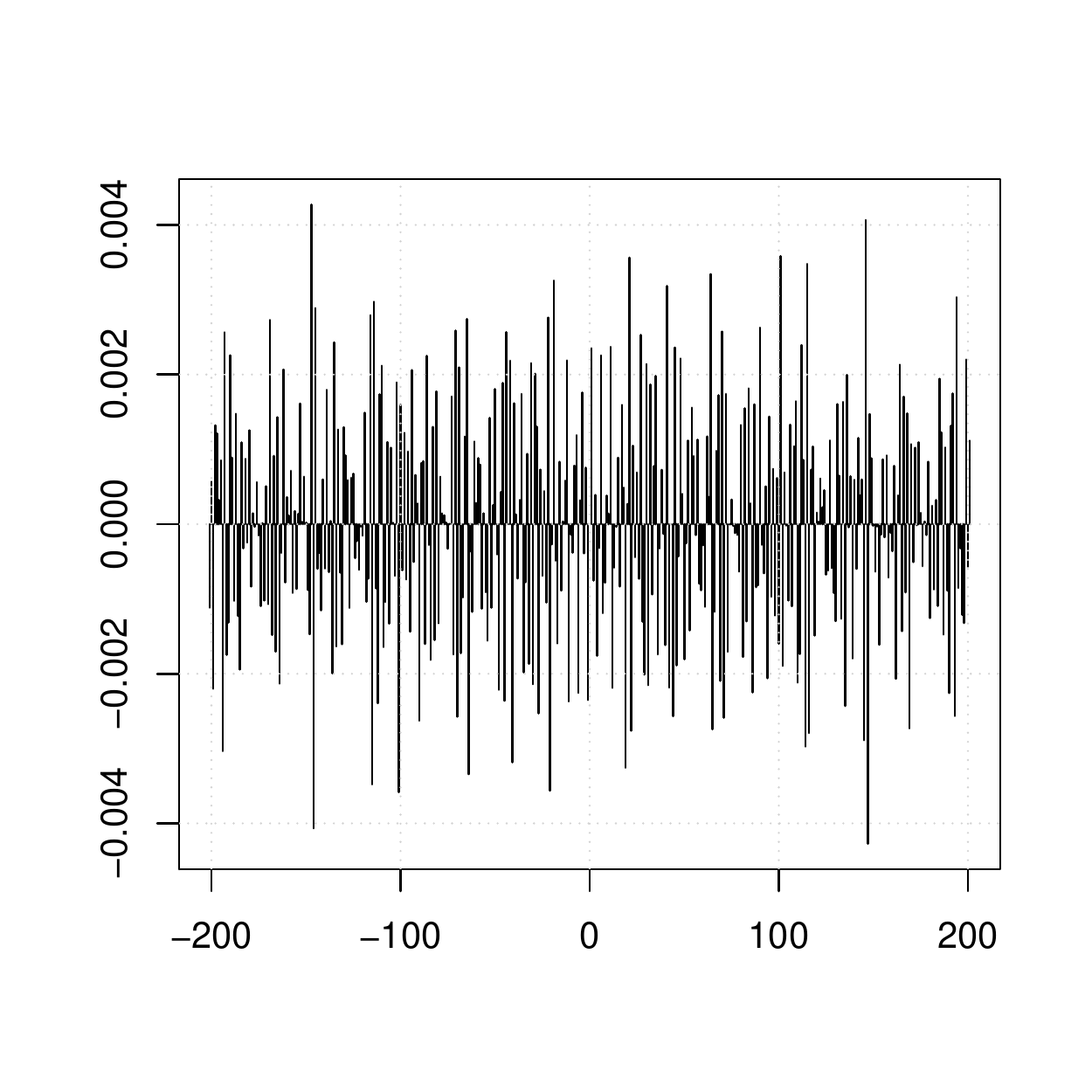}}
\subfigure[$H=0.8$, $\widehat \gamma_{\mR\mI}(j)- \{-\widehat \gamma_{\mI\mR}(j)\}$]{\includegraphics[scale=.6]{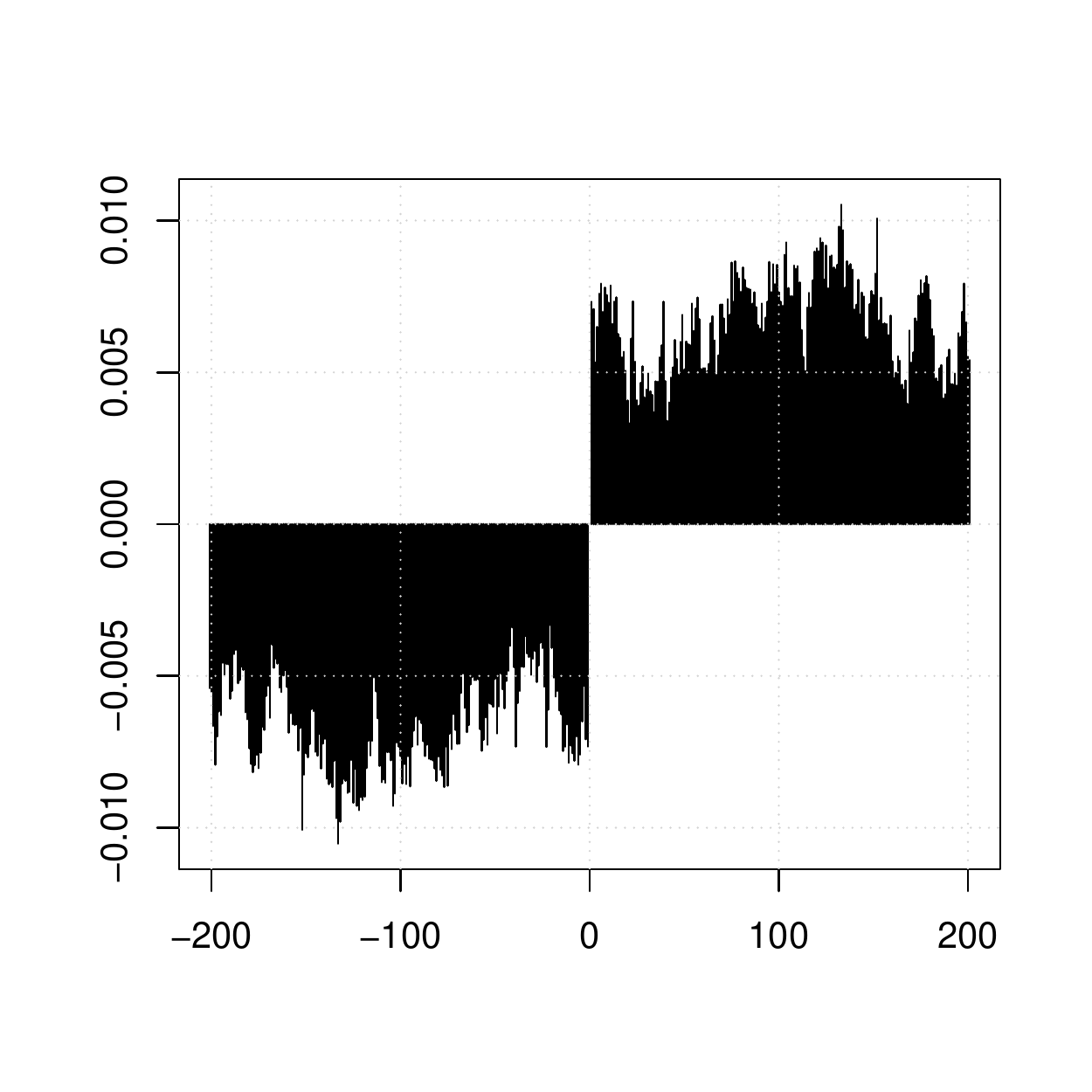}}
\caption{\label{fig:circ} Verification of the circularity property. The computation of empirical cross-covariances are based on a discrete sample path of length $n=10^6$ of a circular complex fBm with variance 1 and $\eta=\frac23|\tan(\pi H)|$.}
\end{figure}

\subsection{Confidence interval for the Hurst exponent of a circular complex fBm}

In this section, we suppose to have access to a sampled version of circular complex fBm. We extend an estimation method to estimate the Hurst exponent and illustrate how the simulation method can be used to derive confidence intervals using parametric Bootstrap samples. Many methods allow to estimate the  self-similarity parameter of a fractional Brownial motion efficiently. We consider here the discrete variations method, see \cite{kent:wood:97,istas:lang:97,coeurjolly:01}. We focus only on the estimation of the Hurst exponent $H$ to illustrate the simulation method. However, we are convinced that using the mentioned papers, estimates for the parameters $\eta$ and $\sigma^2$ can be easily derived.

Let $\ell$ and $q$ be two positive integers.  We consider
the following set of filters ${\cal A}_{\ell,q}$:
\begin{align*}
{\cal A}_{\ell,q}&
= \Big\{ (a_k)_{k\in \Z}: \; a_k=0, \; \forall k\in \Z^{-}\setminus\{0\} \cup \{\ell+1,\dots,\infty\}\\
& \qquad \mbox{ and } \sum_{k\in \Z} k^l a_k=0 , \forall l =0,\ldots, q-1, \;\;  \sum_{k\in \Z} k^q a_k\neq 0  \Big\}.
\end{align*}
Typical examples are the difference filter $\delta_{l,0} -\delta_{l,1}$ and its compositions, Daubechies wavelet filters, and any known wavelet filter with compact support and a sufficient number of vanishing moments. For $a\in {\cal A}_{\ell,q}$ and an integer $\mu\geq1$ we define the $\mu$th dilated version of $a$, say $a^\mu$ as
\begin{align*}
a^\mu_k = \left\{ 
\begin{array}{ll}
   a_{k/\mu}   &  \mbox{if } k\in \mu\Z  \\
    0  &    \mbox{if } k\not\in \mu\Z.
\end{array}\right.
\end{align*}
Apparently, $a^1=a$ and $a^\mu \in {\cal  A}_{\ell,q}$ for any $\mu$. The $\mu$th dilated version is thus simply obtained by oversampling $a$ by a factor of $\mu$, {\it i.e.} by adding $\mu-1$ zeros between  each of the first $\ell+1$ coefficients of the impulse response $a_k$. We denote by $\widetilde Z^\mu$ a discretized sample path of a circular complex fBm $\widetilde Z$ at times $t=0,\dots,n-1$ filtered with $a^\mu$. In other words
\[
  \widetilde Z^\mu(j)= \sum_{k=0}^{\ell} a^\mu_k \widetilde Z(j-k), \quad j=\ell,\dots,n-1.
\]
Let $\mu,\mu^\prime \geq 1$, we denote by $\gamma_{Z^\mu,Z^{\mu^\prime}}(\tau)$ the cross-covariance function between $\widetilde Z^\mu$ and $\widetilde Z^{\mu^\prime}$. 

By definition of $a$, we have
\begin{align*}
 \gamma_{Z^\mu,Z^{\mu^\prime}}(\tau) &= \sum_{q=0}^{\mu\ell}\sum_{r=0}^{\mu^\prime\ell} a_q^\mu a_r^{\mu^\prime} \E \left\{ 
\widetilde Z(\tau+k-q) \widetilde Z(k-r)
 \right\}\\
 &= -\sigma^2 \sum_{q,r=0}^\ell \left\{ 1-\i \eta \,\sign(\tau+\mu^\prime r-\mu q)\right\} |\tau+\mu^\prime r-\mu q|^{2H}.
 \end{align*} 
In particular, we can check that 
\[
  \Var\{\widetilde Z^\mu(j)\}=\gamma_{Z^\mu,Z^{\mu}}(0) = \mu^{2H} \left( 
-\sigma^2 \sum_{q,r=0}^\ell a_qa_r |q-r|^{2H}
  \right).
\]
Now, let $S^2(\mu)$ be the empirical mean squared modulus at scale $\mu$ given by
\[
  S^2(\mu) = \frac{1}{n-\mu\ell} \sum_{j=\mu\ell}^{n-1}|\widetilde Z^\mu(j)|^2.
\]
Since $S^2(\mu)$ is expected to be close to $\kappa \mu^{2H}$ where $\kappa$ is independent of $H$, we propose to estimate $H$ by a linear regression of $\log S^2(\mu)$ on $\log\mu$ for $\mu \in \mathcal M\subset \N^M\setminus\{0\}$, a collection of dilation factors. This estimate is given by
\[
  \widehat H = \frac{\mathbf L^\top}{2\mathbf L^\top\mathbf L} \left\{ \log S^2(\mu)\right\}_{\mu\in \mathcal M} \quad \mbox{ with } \quad \mathbf L = \left(\log \mu -M^{-1} \sum_\mu \log \mu \right)_{\mu\in \mathcal M}.
\]
Such an estimate is very close to the ones proposed by \citet{coeurjolly:01} and~\citet{amblard:coeurjolly:11} to estimate the Hurst exponent of a fBm and the Hurst exponents of a multivariate fBm respectively. We simply exploit the complex characteristic of the process. Using theoretical results from the previous papers, we have the following asymptotic result, given without proof.

\begin{proposition}\label{prop:Hest}
   As $n\to \infty$, $\widehat H$ tends to $H$ with probability 1, and if $p>H+1/4$
   \begin{equation}
     \label{eq:clt}
          \sqrt n (\widehat H-H) \to N\left\{ 0, \frac{\mathbf L^\top \mathbf \Sigma_M \mathbf L}{2 (\mathbf L^\top \mathbf L)^2}\right\},
   \end{equation}
    in distribution, where $\mathbf \Sigma_M$ is the $M\times M$ matrix with entries 
    \begin{equation}
      \label{eq:varAsymp}
     \left( \mathbf \Sigma_M\right)_{\mu\mu^\prime} = \sum_{k\in \Z} \frac{|\gamma_{Z^\mu,Z^{\mu^\prime}}(k)|^2}{\gamma_{Z^\mu,Z^{\mu}}(0)\gamma_{Z^{\mu^\prime},Z^{\mu^\prime}}(0)}, \quad \mu,\mu^\prime \in \mathcal M. 
        \end{equation}
 \end{proposition} 
The condition $p>H+1/4$ is quite standard for such problems and expresses the fact that a circular complex fBm needs to be filtered with a filter with at least two zeroes moments to ensure a Gaussian behaviour for any $H\in (0,1)$. In the rest of this section, we intend to compare several approaches for constructing confidence intervals for $H$. We assume that both parameters $\sigma^2$ and $\eta$ are known. The first approach, referred to as {\sc clt}, uses~\eqref{eq:clt} to construct asymptotic confidence intervals. The series involved in~\eqref{eq:varAsymp} are truncated 
The two other ones are based on parametric Bootstrap. We considered the percentile Bootstrap and Studentized Bootstrap methods, referred to as {\sc ppb} and {\sc spb} respectively, to propose confidence intervals. Given a sample path of a circular complex fBm, we use $2000$ replications of the fitted  model for these parametric Bootstrap methods. Table~\ref{tab:CI} reports the empirical coverage rate and the mean length of 95\% confidence intervals based on 2000 replications of a circular complex fBm for different values of $n, \eta, H$. The variance $\sigma^2$ is set to 1. In terms of coverage rate, the confidence intervals tend to be very comparable. The {\sc pbp} method produces confidence intervals with length larger than the two other methods. Amongst the {\sc clt} and the {\sc spb} approaches, the latter seems to be slightly better in terms of confidence intervals. It is worth noticing that even for small sample sizes, the {\sc clt} method is very competitive.

\begin{center}
\begin{table}[H]
\centering
\begin{tabular}{rllllll}
\hline
& \multicolumn{3}{c}{$H=0.2$} &\multicolumn{3}{c}{$H=0.8$} \\
& {\sc clt} & {\sc ppb} & {\sc spb}& {\sc clt} & {\sc ppb} & {\sc spb}\\
\hline
$n=100$ \\
$\e{1}$& 94 (22.0) & 95 (22.8) & 95 (21.9) & 94 (27.9) & 96 (30.1) & 96 (27.9) \\  
$\e{2}$&  94 (23.7) & 96 (24.4) & 96 (23.6) & 94 (31.2) & 94 (33.3) & 94 (31.1) \\ 
\hline
$n=500$ \\
$\e{1}$&  95 (9.8) & 95 (9.9) & 95 (9.8) & 95 (12.5) & 96 (12.8) & 96 (12.5) \\ 
$\e{2}$&  95 (10.6) & 95 (10.7) & 95 (10.6) & 94 (13.9) & 92 (14.4) & 92 (13.9) \\
\hline
$n=1000$ \\
$\e{1}$&   96 (7.0) & 96 (7.0) & 96 (6.9) & 94 (8.8) & 95 (9.0) & 95 (8.8) \\
$\e2$&   95 (7.5) & 95 (7.5) & 95 (7.5) & 95 (9.9) & 94 (10.1) & 94 (9.8) \\
\hline
\end{tabular}
  \caption{\label{tab:CI} Empirical coverage rate and mean length, between brackets, of 95\% confidence intervals built using~\eqref{eq:clt} (method {\sc clt}) or Bootstrap techniques (methods {\sc ppb} and {\sc spb}). The simulation is based on 2000 replications of circular complex fBm for different sample sizes and different values of $H$ and $\eta$. Empirical coverage rates are reported in percentage and mean lengths are multiplied by 100.
  }
\end{table}
\end{center}

\appendix

\section{Proofs}

\subsection{Auxiliary lemmas}

The following definitions and results are quite standard in Fourier theory. We refer the reader to \citet{zygmund:02}.

\begin{lemma} \label{lem:zygmund} Let $p \in \N \setminus\{0\}$. The Dirichlet and Féjer kernels are respectively defined by 
\begin{align*}
D_p(\omega) &= 1+ 2 \sum_{j=1}^p \cos(2\pi j\omega) = \left\{
\begin{array}{ll}
  \frac{\sin\{\pi \omega(2p+1)\}}{\sin(\pi \omega)} & \mbox{ if } \omega \in \R\setminus \Z\\
  2p+1 & \mbox{ if } \omega \in \Z,
\end{array}
\right. \\
K_p(\omega) &= \sum_{j=0}^p D_j(\omega) = \left\{
\begin{array}{ll}
 \left[ \frac{\sin\{\pi\omega(p+1)\}}{\sin(\pi\omega)}\right]^2 \geq 0,
  & \mbox{ if } \omega \in \R\setminus \Z\\
  (p+1)(2p+1) & \mbox{ if } \omega \in \Z.
\end{array}
\right. 
\end{align*}
The conjugate Dirichlet and Féjer kernels are respectively defined by
\begin{align*}
\widetilde D_p(\omega) &= 2 \sum_{j=0}^p \sin(2\pi j\omega) = \left\{
\begin{array}{ll}
\frac{\cos(\pi \omega)}{\sin(\pi\omega)} - \frac{\cos\{\pi\omega(2p+1)\}}{\sin(\pi\omega)} & \mbox{ if } \omega \in \R\setminus \Z\\
0 & \mbox{ if } \omega \in \Z,\\
\end{array}\right.\\
\widetilde K_p(\omega) &= \sum_{j=0}^p \widetilde D_j(\omega) = \left\{
\begin{array}{ll}
\frac{(p+1)\sin(2\pi\omega)- \sin\{2\pi\omega(p+1)\}}{2\sin(\pi\omega)^2} & \mbox{ if } \omega \in \R\setminus \Z, \\
0 & \mbox{ if }\omega \in \Z.
\end{array}\right.
\end{align*}
Moreover, the conjugate Féjer kernel satisfies $\widetilde K_p(\omega)\geq 0$ whenever  $\omega \in [k,k+1/2]$, for any $k\in \Z$.
\end{lemma}

\begin{proof}
Except for the last result, the proofs can be found in \citet{zygmund:02}. For the last assertion, we need to prove that $ p\sin(t)-\sin(p t)$ is non-negative for $t\in [0,\pi]$ which is proved as follows
\begin{align*}
\sin(pt)&\leq |\sin(pt)|=|\sin\{(p-1)t\}\cos(t)+\cos\{(p-1)t\}\sin(t)| \\
&\leq |\sin\{(p-1)t\}|+|\sin(t)|\leq \dots \leq p |\sin(t)| =p \sin(t)
\end{align*}
when $t\in[0,\pi]$.
\end{proof}

The following result is a summation by parts formula mainly used in the proof of Proposition~\ref{prop:dietrich}. 

\begin{lemma}\label{lem:ipp}
Let $n\geq 1$ and $(f_0,\dots,f_n)^\top$ and $(g_0,\dots,g_n)^\top$ be two vectors of real numbers then,
\begin{equation}
  \label{eq:ipp}
  \sum_{j=0}^n f_j g_j = f_n \sum_{j=0}^n g_j + \sum_{j=0}^{n-1}(f_j-f_{j+1})\sum_{\ell=0}^j
g_k.
\end{equation}
\end{lemma}

\subsection{Proof of Proposition~\ref{prop:step2}}

\begin{proof}
  For $k=0,\dots,\mt-1$, since $({\mt}^{1/2}\bL^{1/2} \bQ^* \bN_\mt)_k$ is a complex normal random variable, we identify it to $\sqrt{\lambda_k}(S_k^\prime + \i T_k^\prime)/\mt$ where $S_k^\prime$ and $T_k^\prime$ are the Gaussian random variables given by
  \[
  S^\prime_k= \sum_{j=0}^{\mt-1}  \cos \left( \frac{2\pi jk}{\mt}\right) (\bN_\mt)_j
  \quad \mbox{ and } \quad
  T^\prime_k= \sum_{j=0}^{\mt-1}  \sin \left( \frac{2\pi jk}{\mt}\right) (\bN_\mt)_j.
  \]
The proof reduces to calculate $\Cov(U_k,U_{k^\prime})$ for $k,k^\prime=0,\dots,\mt-1$ and $U_k=S_k^\prime$ or $T_k^\prime$. 

Let $k,k^\prime\in\{0,\dots,\mt-1\}$. First, by Lemma~\ref{lem:zygmund}, it can be checked that
  \begin{align*}
  \Cov(S_k^\prime,S_{k^\prime}^\prime) &= \sum_{j=0}^{\mt-1} \cos \left(\frac{2\pi jk}{\mt}\right) \cos \left(\frac{2\pi jk^\prime}{\mt}\right) \\
  &= \frac12 \left[ \sum_{j=0}^{\mt-1} \cos\left\{\frac{2\pi j(k-k^\prime)}{\mt}\right\}
 + \sum_{j=0}^{\mt-1}\cos\left\{\frac{2\pi j(k+k^\prime)}{\mt}\right\}
  \right] \\
  &= \frac12 +\frac14 D_{\mt-1}\left(\frac{k-k^\prime}{\mt}\right)+\frac14 D_{\mt-1}\left(\frac{k+k^\prime}{\mt}\right)\\
  &=\left\{ 
  \begin{array}{ll}
    \frac12+ \frac14\{2(\mt-1)\} = \frac{\mt}2 & \mbox{ if } k=k^\prime\\
    \frac12+ \frac14\{2(\mt-1)\} = \frac{\mt}2 & \mbox{ if } k+k^\prime=\mt\\
    \frac12 +\frac14\frac{\sin\{\pi(k-k^\prime)(2\mt-1)/\mt\}}{\sin\{\pi(k-k^\prime)/\mt\}}+ 
    \frac14\frac{\sin\{\pi(k+k^\prime)(2\mt-1)/\mt\}}{\sin\{\pi(k+k^\prime)/\mt\}}=0& \mbox{ otherwise.}
  \end{array}
  \right.
  \end{align*}
  We remark that $k+k^\prime=\mt$ takes place only when $k\wedge k^\prime>0$. Second, with the same ideas
  \begin{align*}
  \Cov(T_k^\prime,T_{k^\prime}^\prime) &= \sum_{j=0}^{\mt-1} \sin \left(\frac{2\pi jk}{\mt}\right) \sin \left(\frac{2\pi jk^\prime}{\mt}\right) \\
  &= \frac12 \left[ \sum_{j=0}^{\mt-1} \cos\left\{\frac{2\pi j(k-k^\prime)}{\mt}\right\}
 - \sum_{j=0}^{\mt-1}\cos\left\{\frac{2\pi j(k+k^\prime)}{\mt}\right\} 
  \right] \\
  &= \frac14 D_{\mt-1} \left( \frac{k-k^\prime}{\mt}\right) -\frac14 D_{\mt-1} \left( \frac{k+k^\prime}{\mt}\right) \\
  &=\left\{ 
\begin{array}{ll}
\frac{\mt}2 & \mbox{ if } k=k^\prime\\
 -\frac{\mt}2 & \mbox{ if } k+k^\prime=\mt\\
 0 & \mbox{ otherwise.}
\end{array}
  \right.
  \end{align*}
  Third,
  \begin{align*}
  \Cov(S_k^\prime,T_{k^\prime}^\prime) &= \sum_{j=0}^{\mt-1} \cos \left(\frac{2\pi jk}{\mt}\right) \sin \left(\frac{2\pi jk^\prime}{\mt}\right) \\
  &= \frac12 \left[ \sum_{j=0}^{\mt-1} \sin\left\{\frac{2\pi j(k+k^\prime)}{\mt}\right\}
 - \sum_{j=0}^{\mt-1}\sin\left\{\frac{2\pi j(k-k^\prime)}{\mt}\right\} 
  \right] \\
  &= \frac14 \widetilde D_{\mt-1} \left( \frac{k+k^\prime}{\mt}\right) -\frac14 \widetilde D_{\mt-1} \left( \frac{k-k^\prime}{\mt}\right) \\
  &=0,
  \end{align*}
  whereby we deduce the result.
\end{proof}

\subsection{Proof of Proposition~\ref{prop:H}}

\begin{proof}
 It is clear that $\mathbf H = \bQ \Lambda^{1/2} (\bQ^*)^2 \Lambda^{1/2} \bQ$. Using, the proof of Proposition~\ref{prop:step2}, we can check that for $j,k=0,\dots,\mt-1$
 \begin{align*}
  (\bQ^*)_{jk}^2 &= \mt^{-1}\sum_{\ell=0}^{\mt-1} \left\{\cos \left(\frac{2\pi j\ell}{\mt}\right) \cos \left(\frac{2\pi k\ell}{\mt}\right)  
- \sin \left(\frac{2\pi j\ell}{\mt}\right) \sin \left(\frac{2\pi k\ell}{\mt}\right)  
 \right\} \\
 &\;\; + \i \,\mt^{-1}
 \sum_{\ell=0}^{\mt-1} \left\{\sin \left(\frac{2\pi j\ell}{\mt}\right) \cos \left(\frac{2\pi k\ell}{\mt}\right) + 
\cos \left(\frac{2\pi j\ell}{\mt}\right) \sin \left(\frac{2\pi k\ell}{\mt}\right)
  \right\} \\
  &= \left\{ 
\begin{array}{ll}
1 & \mbox{ if } j+k=\mt \\ 0 & \mbox{ otherwise.}
\end{array}
  \right.
 \end{align*}
 Let $\mathbf L$ be the $\mt\times \mt$ matrix given by $(\mathbf L)_{jk}=\sqrt{\lambda_j \lambda_{\mt-j}}$, if $j\wedge k>0$ and $j+k=\mt$ and 0 otherwise. We have, $\bL^{1/2}(\bQ^*)^2\bL^{1/2}=\mathbf L$. The result follows from
 \begin{align*}
  (\bQ \mathbf L)_{jk} &= \mt^{-1/2}\sum_{\ell=0}^{\mt-1} e^{-\frac{2\i \pi j\ell}{\mt}} \mathbf (L)_{\ell k} = \mt^{-1/2} e^{-\frac{2\i \pi j(\mt-k)}\mt} \sqrt{\lambda_k \lambda_{\mt-k}} \mathbf 1(k>0) = (\bQ^* \mathbf V)_{jk},
\end{align*} 
where $\mathbf V=\mathrm{diag}(v_k, k=0,\dots,\mt-1)$ is the diagonal matrix with elements given by $v_0=0$ and $v_k= \sqrt{\lambda_k \lambda_{\mt-k}}$ for $k\geq 1$.
\end{proof}

\subsection{Proof of Proposition~\ref{prop:step2algorithm2}}

\begin{proof}
We proceed similarly to the proof of Proposition~\ref{prop:step2}. We identify the Gaussian random variable $W_k$ to $\sqrt{\lambda_k/2}(S_k^\prime + \i T_k^\prime)/\mt$ where $S_k^\prime$ and $T_k^\prime$ are the Gaussian random variables given by
  \begin{align*}
  S^\prime_k&= \sum_{j=0}^{\mt-1}  \left\{
  \cos \left( \frac{2\pi jk}{\mt}\right) (\bN_\mt)_{1,j}
  -\sin \left( \frac{2\pi jk}{\mt}\right) (\bN_\mt)_{2,j}
  \right\} \\
T^\prime_k&= \sum_{j=0}^{\mt-1}  \left\{
\sin \left( \frac{2\pi jk}{\mt}\right) (\bN_\mt)_{1,j}
+\cos \left( \frac{2\pi jk}{\mt}\right) (\bN_\mt)_{2,j}.
\right\}.
\end{align*}
We leave the reader to check that for any $k,k^\prime \in \{0,\dots,\mt-1\}$, $\Cov(S_k^\prime,S_{k^\prime}^\prime)=\Cov(T_k^\prime,T_{k^\prime}^\prime) = \mt \delta_{kk^\prime}$ and $\Cov(S_k^\prime,T_{k^\prime}^\prime)=0$, whereby we deduce the result.
\end{proof}

\subsection{Proof of Proposition~\ref{prop:approximation}}

\begin{proof}
 For each $x>0$,
 \begin{align*}
  \PP\left(\max_{j=0,\dots,n-1} \sigma_j^{-1} |\Delta_j| > x\right)&=
  1- \PP \left( \bigcap_{j=0}^{n-1} \{\sigma_j^{-1}|\Delta_j|\leq x\} \right) \\
  &\leq 1- \PP \left[ \bigcap_{j=0}^{n-1} \left\{ 
  |\mathrm{Re}(\Delta_j)| \leq \frac{x \sigma_j}{\sqrt{2}} ,
  |\mathrm{Im}(\Delta_j)| \leq \frac{x \sigma_j}{\sqrt{2}}  
\right\}
  \right].
  \end{align*} 
From~\citet{dunn:58,dunn:59}
\[
  \PP\left(\max_{j=0,\dots,n-1} \sigma_j^{-1} |\Delta_j| > x\right) \leq 1- \prod_{j=0}^{n-1} 
  \PP \left\{ |\mathrm{Re}(\Delta_j)| \leq  \frac{x \sigma_j}{\sqrt{2}} \right\}
  \PP \left\{ |\mathrm{Im}(\Delta_j)| \leq \frac{x \sigma_j}{\sqrt{2}} \right\}
\]
whereby we deduce the result.
\end{proof}

\subsection{Proof of Proposition~\ref{prop:craigmile}}

\begin{proof}
(i) For any $k=0,\dots,\mt-1$, we have 
\begin{align*}
\lambda_k &\geq \gamma(0) + 2 \sum_{j=1}^m \{\mR(j) -s \mI(j) \}= \sum_{|j|\leq m} \{\mR(j) - s \,\sign(j) \mI(j)\} =:A_m. 
\end{align*}
Since $A_{m+1}-A_m= 2\mR(m+1) -s \{\mI(m+1) - \mI(-m-1)\}\leq 0$, $\{A_m\}_{m\geq 1}$ is a decreasing sequence. By Assumption, $\mathbf S(\omega)= \sum_{j\in \Z} \mathbf M(j)e^{-2\i \pi j \omega}$ is a Hermitian non-negative definite matrix for every $\omega$. In particular, for  $y=(1,-s)^\top$ and $\omega=0$, we have
\[
  y^\top \mathbf S(0) y = \sum_{j\in \Z} \{\mR(j) -s \mI(j) \} \geq0,
\]
whereby we deduce that $\lambda_k \geq \lim_{m\to \infty} y^\top \mathbf S(0)y \geq0$.\\
(ii) In this setting, $\mR(j)=2\gamma_\mR(j)$ and $\mI(j)=2\gamma_{\mR\mI}(j)=2\eta \sign(j)\gamma_{\mR}(j)$. Hence, the matrix $\mathbf M(j)$ given by~\eqref{eq:M} takes the form
\[
   \mathbf M(j)= \gamma_{\mR}(j)\left(
\begin{array}{ll}
1 & \eta \\
\eta & 1 
\end{array}
  \right).
  \] 
This indeed  corresponds to the covariance matrix of a stationary bivariate process, say $\check Z$. By the same argument than \citet{craigmile:03}[Proposition~1] in the real case, if $\gamma_\mR$ is not summable then the spectrum at zero frequency is negative infinity, an impossibility. Hence, $\check Z$ admits a well--defined spectral density matrix and Proposition~\ref{prop:craigmile} applies.\\
(iii) Let $\tilde \phi= \phi \mt$. Using standard trigonometric identities, the expression of $\lambda_k$ reduces to
\begin{equation}\label{eq:lambdak3}
  \lambda_k  = r(0) + 2 \sum_{j=1}^m r(j) \cos\left\{ \frac{2\pi j}{\mt} (k+\tilde \phi)\right\}. 
\end{equation}
The rest of the proof follows the same lines as for the proof of~(i). 
\end{proof}

\subsection{Proof of Proposition~\ref{prop:dietrich}}

\begin{proof}
Lemmas~\ref{lem:zygmund}--\ref{lem:ipp} are  used in this proof. By the summation by parts formula, we have
\begin{align*}
 2 \sum_{j=0}^{\mt-1} \mR(j) \cos \left( \frac{2\pi jk}{\mt}\right) &= \mR(m) \left\{D_m \left(\frac k\mt\right)+1\right\} + \sum_{j=0}^{\mt-1} \Delta \mR(j)\left\{ 
D_j\left( \frac k\mt\right)+1
 \right\} \\
 2 \sum_{j=1}^{\mt-1} \mI(j) \sin \left( \frac{2\pi jk}{\mt}\right) &= \mI(m) \widetilde D_m \left(\frac k\mt\right) + \sum_{j=1}^{\mt-1} \Delta \mI(j) 
\widetilde D_j\left( \frac k\mt\right). \\
 \end{align*} 
Reinjecting these equations in~\eqref{eq:lambdak2} leads to
\[
  \lambda_k = -\mI(m) \widetilde D_m\left(\frac{k}\mt\right)+ \sum_{j=1}^{m-1} (-\Delta \mI)(j) \widetilde D_j\left(\frac k\mt\right)  + \sum_{j=0}^{m-1} \Delta \mR(j) D_j\left(\frac k\mt\right).
\]
Another application of Lemma~\ref{lem:ipp} allows us to obtain~\eqref{eq:lambdakIPP}. \\
(i) This assertion ensues from Lemma~\ref{lem:zygmund} and condition~\eqref{eq:conddietrich}.\\
(ii) This point is a particular case of (i).\\
\end{proof}

\subsection{Proof of Proposition~\ref{prop:dietrichMod}}

\begin{proof}
  Starting from \eqref{eq:lambdak3}, if we apply twice a summation by parts formula, we obtain
\begin{align}
  \lambda_k &= 
   \Delta r(m-1) K_{m-1}\left(\frac{k+\tilde \phi}\mt\right) + \sum_{j=0}^{m-2}(\Delta^2 r)(j) K_j\left( \frac{k+\tilde\phi}\mt\right), \label{eq:lambdakIPP2}
\end{align}
which is non-negative by assumption and Lemma~\ref{lem:zygmund}.
\end{proof}

\bibliographystyle{royal} 
\bibliography{complex} 
 
\end{document}